\documentclass[11pt]{preprint}
\usepackage[full]{textcomp}
\usepackage[osf]{newtxtext} 
\usepackage{colortbl}

\usepackage{comment}
\usepackage[usenames,dvipsnames]{xcolor}

\usepackage{amssymb}
\usepackage{lmodern}
\usepackage{mathtools}

\usepackage{hyperref}
\usepackage{breakurl}
\usepackage{float}
\usepackage{aligned-overset}
\usepackage{mhenvs}
\usepackage{mhequ} 
\newcommand{\be}{\begin{equation*}}
\newcommand{\ee}{\end{equation*}}
\usepackage{mhsymb}
\usepackage{booktabs}
\usepackage{tikz,tkz-tab}
\tikzset{every picture/.style={line width=0.75pt}} 
\usetikzlibrary{matrix,decorations.pathreplacing, calc, positioning,fit,arrows.meta}

\usepackage{pgfplots}
\pgfplotsset{compat=1.18}
\usetikzlibrary{babel}

\usepackage{tcolorbox}
\usepackage{mathrsfs}
\usepackage[utf8]{inputenc}
\usepackage{longtable}
\usepackage{wrapfig}
\usepackage{subcaption}
\usepackage{mathrsfs}
\usepackage{epsfig}
\usepackage{microtype}
\usepackage{comment}
\usepackage{wasysym}
\usepackage{centernot}
\usepackage{enumitem}
\usepackage{bm}
\usepackage{stackrel}
\usepackage{graphicx}
\makeatletter
\newcommand{\globalcolor}[1]{%
  \color{#1}\global\let\default@color\current@color
}
\makeatother

\usetikzlibrary{calc}
\usetikzlibrary{decorations}
\usetikzlibrary{positioning}
\usetikzlibrary{shapes}
\usetikzlibrary{external}

\definecolor{blush}{rgb}{0.87, 0.36, 0.51}
	\definecolor{brightcerulean}{rgb}{0.11, 0.67, 0.84}
	\definecolor{greenryb}{rgb}{0.4, 0.69, 0.2}

\newif\ifdark
\darkfalse

\ifdark
\definecolor{darkred}{rgb}{0.9,0.2,0.2}
\definecolor{darkblue}{rgb}{0.7,0.3,1}
\definecolor{darkgreen}{rgb}{0.1,0.9,0.1}
\definecolor{franck}{rgb}{0,0.8,1}
\definecolor{pagebackground}{rgb}{.15,.21,.18}
\definecolor{pageforeground}{rgb}{.84,.84,.85}
\pagecolor{pagebackground}
\AtBeginDocument{\globalcolor{pageforeground}}
\definecolor{symbols}{rgb}{0,0.7,1}
\colorlet{connection}{red!80!black}
\colorlet{boxcolor}{blue!50}

\else

\definecolor{darkred}{rgb}{0.7,0.1,0.1}
\definecolor{darkblue}{rgb}{0.4,0.1,0.8}
\definecolor{darkgreen}{rgb}{0.1,0.7,0.1}
\definecolor{franck}{rgb}{0,0,1}
\definecolor{pagebackground}{rgb}{1,1,1}
\definecolor{pageforeground}{rgb}{0,0,0}
\colorlet{symbols}{blue!90!black}
\colorlet{connection}{red!30!black}
\colorlet{boxcolor}{blue!50!black}

\fi

\def\slash{\leavevmode\unskip\kern0.18em/\penalty\exhyphenpenalty\kern0.18em}
\def\dash{\leavevmode\unskip\kern0.18em--\penalty\exhyphenpenalty\kern0.18em}

\DeclareMathAlphabet{\mathbbm}{U}{bbm}{m}{n}

\DeclareFontFamily{U}{BOONDOX-calo}{\skewchar\font=45 }
\DeclareFontShape{U}{BOONDOX-calo}{m}{n}{
  <-> s*[1.05] BOONDOX-r-calo}{}
\DeclareFontShape{U}{BOONDOX-calo}{b}{n}{
  <-> s*[1.05] BOONDOX-b-calo}{}
\DeclareMathAlphabet{\mcb}{U}{BOONDOX-calo}{m}{n}
\SetMathAlphabet{\mcb}{bold}{U}{BOONDOX-calo}{b}{n}

\setlist{noitemsep,topsep=4pt,leftmargin=1.5em}

\DeclareMathAlphabet{\mathbbm}{U}{bbm}{m}{n}

\DeclareMathAlphabet{\mcb}{U}{BOONDOX-calo}{m}{n}
\SetMathAlphabet{\mcb}{bold}{U}{BOONDOX-calo}{b}{n}
\DeclareFontFamily{U}{mathx}{\hyphenchar\font45}
\DeclareFontShape{U}{mathx}{m}{n}{
      <5> <6> <7> <8> <9> <10>
      <10.95> <12> <14.4> <17.28> <20.74> <24.88>
      mathx10
      }{}
\DeclareSymbolFont{mathx}{U}{mathx}{m}{n}
\DeclareMathSymbol{\bigtimes}{1}{mathx}{"91}

\providecommand{\figures}{false}
{ \ifthenelse{\equal{\figures}{false}} {#1}{\[ {\rm Figure \ missing !} \]} }{}

\usepackage{mathtools}

\tikzstyle{tinydots}=[dash pattern=on \pgflinewidth off \pgflinewidth]
\tikzstyle{superdense}=[dash pattern=on 4pt off 1pt]

\newcommand{\beq}{\begin{equation}}
\newcommand{\eeq}{\end{equation}}

\usepackage{empheq}

\def\${|\!|\!|}

\newcounter{theorem}

\newtheorem{ex}[theorem]{Example}

\newenvironment{DIFnomarkup}{}{} 

\theorembodyfont{\rmfamily}

\newfont{\indic}{bbmss12}

\def\Nabla_#1{\nabla_{\!#1}}

\makeatletter
\pgfdeclareshape{crosscircle}
{
  \inheritsavedanchors[from=circle] 
  \inheritanchorborder[from=circle]
  \inheritanchor[from=circle]{north}
  \inheritanchor[from=circle]{north west}
  \inheritanchor[from=circle]{north east}
  \inheritanchor[from=circle]{center}
  \inheritanchor[from=circle]{west}
  \inheritanchor[from=circle]{east}
  \inheritanchor[from=circle]{mid}
  \inheritanchor[from=circle]{mid west}
  \inheritanchor[from=circle]{mid east}
  \inheritanchor[from=circle]{base}
  \inheritanchor[from=circle]{base west}
  \inheritanchor[from=circle]{base east}
  \inheritanchor[from=circle]{south}
  \inheritanchor[from=circle]{south west}
  \inheritanchor[from=circle]{south east}
  \inheritbackgroundpath[from=circle]
  \foregroundpath{
    \centerpoint%
    \pgf@xc=\pgf@x%
    \pgf@yc=\pgf@y%
    \pgfutil@tempdima=\radius%
    \pgfmathsetlength{\pgf@xb}{\pgfkeysvalueof{/pgf/outer xsep}}%
    \pgfmathsetlength{\pgf@yb}{\pgfkeysvalueof{/pgf/outer ysep}}%
    \ifdim\pgf@xb<\pgf@yb%
      \advance\pgfutil@tempdima by-\pgf@yb%
    \else%
      \advance\pgfutil@tempdima by-\pgf@xb%
    \fi%
    \pgfpathmoveto{\pgfpointadd{\pgfqpoint{\pgf@xc}{\pgf@yc}}{\pgfqpoint{-0.707107\pgfutil@tempdima}{-0.707107\pgfutil@tempdima}}}
    \pgfpathlineto{\pgfpointadd{\pgfqpoint{\pgf@xc}{\pgf@yc}}{\pgfqpoint{0.707107\pgfutil@tempdima}{0.707107\pgfutil@tempdima}}}
    \pgfpathmoveto{\pgfpointadd{\pgfqpoint{\pgf@xc}{\pgf@yc}}{\pgfqpoint{-0.707107\pgfutil@tempdima}{0.707107\pgfutil@tempdima}}}
    \pgfpathlineto{\pgfpointadd{\pgfqpoint{\pgf@xc}{\pgf@yc}}{\pgfqpoint{0.707107\pgfutil@tempdima}{-0.707107\pgfutil@tempdima}}}
  }
}
\makeatother

\def\symbol#1{\textcolor{symbols}{#1}}

\def\decorate#1#2{
        \ifnum#2>0
    		\foreach \count in {1,...,#2}{
	       	let
				\p1 = (sourcenode.center),
                \p2 = (sourcenode.east),
				\n1 = {\x2-\x1},
				\n2 = {1mm},
				\n3 = {(1.3+0.6*(\count-1))*\n1},
				\n4 = {0.7*\n1}
			in 
        		node[rectangle,fill=symbols,rotate=30,inner sep=0pt,minimum width=0.2*\n2,minimum height=\n2] at ($(sourcenode.center) + (\n3,\n4)$) {}
				}
		\fi
        \ifnum#1>0
    		\foreach \count in {1,...,#1}{
	       	let
				\p1 = (sourcenode.center),
                \p2 = (sourcenode.east),
				\n1 = {\x2-\x1},
				\n2 = {1mm},
				\n3 = {(1.3+0.6*(\count-1))*\n1},
				\n4 = {0.7*\n1}
			in 
        		node[rectangle,fill=symbols,rotate=-30,inner sep=0pt,minimum width=0.2*\n2,minimum height=\n2] at ($(sourcenode.center) + (-\n3,\n4)$) {}
				}
		\fi
}

\tikzset{
    dectriangle/.style 2 args={
        triangle,
        alias=sourcenode,
        append after command={\decorate{#1}{#2}}
    },
    dectriangle/.default={0}{0},
}

\tikzset{
	cross/.style={path picture={ 
  		\draw[symbols]
			(path picture bounding box.south east) -- (path picture bounding box.north west) (path picture bounding box.south west) -- (path picture bounding box.north east);
		}},
root/.style={circle,fill=green!50!black,inner sep=0pt, minimum size=1.2mm},
        dot/.style={circle,fill=pageforeground,inner sep=0pt, minimum size=1mm},
        dotred/.style={circle,fill=pageforeground!50!pagebackground,inner sep=0pt, minimum size=2mm},
        var/.style={circle,fill=pageforeground!10!pagebackground,draw=pageforeground,inner sep=0pt, minimum size=3mm},
        kernel/.style={semithick,shorten >=2pt,shorten <=2pt},
        kernels/.style={snake=zigzag,shorten >=2pt,shorten <=2pt,segment amplitude=1pt,segment length=4pt,line before snake=2pt,line after snake=5pt,},
        rho/.style={densely dashed,semithick,shorten >=2pt,shorten <=2pt},
           testfcn/.style={dotted,semithick,shorten >=2pt,shorten <=2pt},
        renorm/.style={shape=circle,fill=pagebackground,inner sep=1pt},
        labl/.style={shape=rectangle,fill=pagebackground,inner sep=1pt},
        xic/.style={very thin,circle,draw=symbols,fill=symbols,inner sep=0pt,minimum size=1.2mm},
        g/.style={very thin,rectangle,draw=symbols,fill=symbols!10!pagebackground,inner sep=0pt,minimum width=2.5mm,minimum height=1.2mm},
        xi/.style={very thin,circle,draw=symbols,fill=symbols!10!pagebackground,inner sep=0pt,minimum size=1.2mm},
	xies/.style={very thin,rectangle,fill=green!50!black!25,draw=symbols,inner sep=0pt,minimum size=1.1mm},
	xiesf/.style={very thin,rectangle,fill=green!50!black,draw=symbols,inner sep=0pt,minimum size=1.1mm},
        xix/.style={very thin,crosscircle,fill=symbols!10!pagebackground,draw=symbols,inner sep=0pt,minimum size=1.2mm},
        X/.style={very thin,cross,rectangle,fill=pagebackground,draw=symbols,inner sep=0pt,minimum size=1.2mm},
	xib/.style={thin,circle,fill=symbols!10!pagebackground,draw=symbols,inner sep=0pt,minimum size=1.6mm},
	xie/.style={thin,circle,fill=green!50!black,draw=symbols,inner sep=0pt,minimum size=1.6mm},
	xid/.style={thin,circle,fill=symbols,draw=symbols,inner sep=0pt,minimum size=1.6mm},
	xibx/.style={thin,crosscircle,fill=symbols!10!pagebackground,draw=symbols,inner sep=0pt,minimum size=1.6mm},
	kernels2/.style={very thick,draw=connection,segment length=12pt},
	keps/.style={thin,draw=symbols,->},
	kepspr/.style={thick,draw=connection,->},
	krho/.style={thin,draw=symbols,superdense,->},
	krhopr/.style={thick,draw=connection,superdense},
	triangle/.style = { regular polygon, regular polygon sides=3},
	not/.style={thin,circle,draw=connection,fill=connection,inner sep=0pt,minimum size=0.5mm},
	diff/.style = {very thin,draw=symbols,triangle,fill=red!50!black,inner sep=0pt,minimum size=1.6mm},
	diff1/.style = {very thin,dectriangle={1}{0},fill=red!50!black,draw=symbols,inner sep=0pt,minimum size=1.6mm},
	diff2/.style = {very thin,dectriangle={1}{1},fill=red!50!black,draw=symbols,inner sep=0pt,minimum size=1.6mm},
		diffmini/.style = {very thin,rectangle,fill=black,draw=black,inner sep=0pt,minimum size=0.75mm},
	 kernelsmod/.style={very thick,draw=connection,segment length=12pt},
	 rec/.style = {very thin,rectangle,fill=black,draw=black,inner sep=0pt,minimum size=2mm},
	cerc/.style={very thin,circle,draw=black,fill=symbols,inner sep=0pt,minimum size=2mm},
	stars/.style={very thin,star,star points=6,star point ratio=0.5, draw=black,fill=red,inner sep=0pt,minimum size=0.7mm},
	>=stealth,
        }
        \tikzset{
root/.style={circle,fill=black!50,inner sep=0pt, minimum size=3mm},
        circ/.style={circle,fill=white,draw=black,very thin,inner sep=.5pt, minimum size=1.2mm},
        round1/.style={fill=white,outer sep = 0,inner sep=2pt,rounded corners=1mm,draw,text=black,thin,minimum size=1.2mm},
          circ1/.style={circle,fill=red!10,draw=red,very thin,inner sep=.5pt, minimum size=1.2mm},
        rect/.style={fill=white,outer sep = 0,inner sep=2pt,rectangle,draw,text=black,thin,minimum size=1.2mm},
        rect1/.style={fill=white,outer sep = 0,inner sep=2pt,rectangle,draw,text=black,thin,minimum size=1.2mm},
        round2/.style={fill=red!10,outer sep = 0,inner sep=2pt,rounded corners=1mm,draw,text=black,thin,minimum size=1.2mm},
       round3/.style={fill=blue!10,outer sep = 0,inner sep=2pt,rounded corners=1mm,draw,text=black,thin,minimum size=1.2mm}, 
        rect2/.style={fill=black!10,outer sep = 0,inner sep=2pt,rectangle,draw,text=black,thin,minimum size=1.2mm},
        dot/.style={circle,fill=black,inner sep=0pt, minimum size=1.2mm},
        dotred/.style={circle,fill=black!50,inner sep=0pt, minimum size=2mm},
        var/.style={circle,fill=black!10,draw=black,inner sep=0pt, minimum size=3mm},
        kernel/.style={semithick,shorten >=2pt,shorten <=2pt},
         diag/.style={thin,shorten >=4pt,shorten <=4pt},
        kernel1/.style={thick},
        kernels/.style={snake=zigzag,shorten >=2pt,shorten <=2pt,segment amplitude=1pt,segment length=4pt,line before snake=2pt,line after snake=5pt,},
		kernels1/.style={snake=zigzag,segment amplitude=0.5pt,segment length=2pt},
		rho1/.style={densely dotted,semithick},
        rho/.style={densely dashed,semithick,shorten >=2pt,shorten <=2pt},
           testfcn/.style={dotted,semithick,shorten >=2pt,shorten <=2pt},
           visible/.style={draw, circle, fill, inner sep=0.25ex},
        renorm/.style={shape=circle,fill=white,inner sep=1pt},
        labl/.style={shape=rectangle,fill=white,inner sep=1pt},
        xic/.style={very thin,circle,fill=symbols,draw=black,inner sep=0pt,minimum size=1.2mm},
        xi/.style={very thin,circle,fill=blue!10,draw=black,inner sep=0pt,minimum size=1.2mm},
	xib/.style={very thin,circle,fill=blue!10,draw=black,inner sep=0pt,minimum size=1.6mm},
	xie/.style={very thin,circle,fill=green!50!black,draw=black,inner sep=0pt,minimum size=1mm},
	xid/.style={very thin,circle,fill=symbols,draw=black,inner sep=0pt,minimum size=1.6mm},
	edgetype/.style={very thin,circle,draw=black,inner sep=0pt,minimum size=5mm},
	nodetype/.style={very thick,circle,draw=black,inner sep=0pt,minimum size=5mm},
	kernels2/.style={very thick,draw=connection,segment length=12pt},
clean/.style={thin,circle,fill=black,inner sep=0pt,minimum size=1mm},	not/.style={thin,circle,fill=symbols,draw=connection,fill=connection,inner sep=0pt,minimum size=0.8mm},
	>=stealth,
        }

\makeatletter
\def\DeclareSymbol#1#2#3{%
	\expandafter\gdef\csname MH@symb@#1\endcsname{\tikzsetnextfilename{symbol#1}%
	\tikz[baseline=#2,scale=0.15,draw=symbols,line join=round]{#3}}%
	\expandafter\gdef\csname MH@symb@#1s\endcsname{\scalebox{0.75}{\tikzsetnextfilename{symbol#1}%
	\tikz[baseline=#2,scale=0.15,draw=symbols,line join=round]{#3}}}%
	\expandafter\gdef\csname MH@symb@#1ss\endcsname{\scalebox{0.65}{\tikzsetnextfilename{symbol#1}%
	\tikz[baseline=#2,scale=0.15,draw=symbols,line join=round]{#3}}}%
	}
\def\<#1>{\ifthenelse{\boolean{mmode}}{\mathchoice{\csname MH@symb@#1\endcsname}{\csname MH@symb@#1\endcsname}{\csname MH@symb@#1s\endcsname}{\csname MH@symb@#1ss\endcsname}}{\csname MH@symb@#1\endcsname}}
\makeatother

 \def\1{\mathbf{\symbol{1}}}

\DeclareMathAlphabet{\mathpzc}{OT1}{pzc}{m}{it}

\def\eqref#1{(\ref{#1})}

\makeatletter 
\newcommand*{\bigcdot}{}
\DeclareRobustCommand*{\bigcdot}{%
  \mathbin{\mathpalette\bigcdot@{}}%
}
\newcommand*{\bigcdot@scalefactor}{.5}
\newcommand*{\bigcdot@widthfactor}{1.15}
\newcommand*{\bigcdot@}[2]{%
  \sbox0{$#1\vcenter{}$}
  \sbox2{$#1\cdot\m@th$}%
  \hbox to \bigcdot@widthfactor\wd2{%
    \hfil
    \raise\ht0\hbox{%
      \scalebox{\bigcdot@scalefactor}{%
        \lower\ht0\hbox{$#1\bullet\m@th$}%
      }%
    }%
    \hfil
  }%
}
\makeatother

\def\act{\bigcdot}

\tcbset
{colframe=boxcolor,colback=symbols!7!pagebackground,coltext=pageforeground,
fonttitle=\bfseries,nobeforeafter,center title,size=fbox,boxsep=1.5pt,
top=0mm,bottom=0mm,boxsep=0mm,tcbox raise base}

\def\two{{\<generic>\kern0.05em\<genericb>}}
\def\twoI{{\<Ito>\kern0.05em\<Itob>}}

\def\mail#1{\burlalt{#1}{mailto:#1}}

\usepackage{thmtools}

\begin{document}

\title{Kruskal-style algorithm for cubic Schrödinger equation molecule reduction}

\author{Yvain Bruned, Valentin Clarisse}
\institute{ 
	Universite de Lorraine, CNRS, IECL, F-54000 Nancy, France
	\\
	Email:\ \begin{minipage}[t]{\linewidth}
		\mail{yvain.bruned@univ-lorraine.fr}
		\\
		\mail{valentin.clarisse@univ-lorraine.fr}.
\end{minipage}}

\maketitle 

\begin{abstract}
We are interested in the molecule reduction algorithm introduced by Deng and Hani. They use this algorithm to establish a rigidity theorem, which plays a central role in the kinetic-time derivation of the wave equation associated with the cubic Schrödinger equation. In the present article, we show that this algorithm is a graph traversal algorithm of Kruskal type, and we prove that it constructs a Kruskal spanning tree of the input molecule. This reveals the origin of the main tool for deriving kinetic equations which has also been used for the long time derivation of the Boltzmann equation.
\end{abstract}

\setcounter{tocdepth}{2}
\tableofcontents

\section{Introduction}

Let $d \geqslant 3$ be an integer, we consider the following cubic Schrödinger equation:
\begin{equs}
	\label{NLS_equa}
\begin{cases}
	(i \partial_t-\Delta_\beta) u = - \lambda^2 u |u|^2 & \text{over } \mathbb{R}^+ \times \mathbb{T}_L^d \\
	u(0,\cdot) = u_{\textup{in}}
\end{cases}
\end{equs}
with:
\begin{itemize}
	\item $L\gg 1$ the size of the torus~;
	\item $\displaystyle\Delta_\beta = \dfrac{1}{2\pi} \sum_{k=1}^d \beta_k \partial_k^2$ the Laplacian "twisted" by a vector $\beta \in (\mathbb{R}_+)^d \setminus \mathcal{N}$, where $\mathcal{N}$ is a negligible subset of $\mathbb{R}^d$ with respect to the Lebesgue measure (see \cite[Lemma A.8]{DH23}).
	\item $\widehat{u_{\text{in}}}(k)=\sqrt{\phi_{\text{in}}(k)}X_k$ where $\phi_{\text{in}}\in\mathcal{S}(\mathbb{R}^d,\mathbb{R}^+)$ and $(X_k)_{k\in\mathbb{Z}_L^d}$ are i.i.d. Gaussian random variables satisfying
	\begin{equs}
		\mathbb{E}(X_k \bar{X}_{\ell}) = \delta_{\ell,k}, \quad \mathbb{E}(X_k X_{\ell}) = 0.
	\end{equs} 
\end{itemize}
As explained in section 2.1 of \cite{DH21}, the previous equation can be reformulated as followed:
\[
\begin{cases}
	\partial_t\mathbf{a} = \mathcal{C}(\mathbf{a}, \overline{\mathbf{a}}, \mathbf{a})\\
	\mathbf{a}(0) = \left(\sqrt{\phi_{\text{in}}(k)} X_k\right)_{k\in\mathbb{Z}_L^d}
\end{cases}
\]
where $\mathbf{a}=(a_k)_{k\in\mathbb{Z}_L^d}$ and $\mathcal{C}$ is a trilinear operator. In \cite{DH23}, the authors manage to obtain the  rigorous derivation of the wave kinetic equation from \eqref{NLS_equa} at the kinetic timescale in \cite[Theorem 1.1]{DH23} that we recall below

\begin{theorem}[Deng--Hani]
	\label{main_theorem_DH23}
	For all $\beta \in (\mathbb{R}^+)^d \setminus \mathcal{N}$,  $A \geqslant 40d$, $\phi_{\text{in}} \in \mathcal{S}(\mathbb{R}^d, \mathbb{R}^+)$, there exists $\delta > 0$ such that, for $L$ sufficiently large satisfying the scaling law $\lambda^2 = L^{d-1}$, the cubic Schrödinger equation admits a smooth solution up to time $T = \delta T_{\textup{kin}}=\delta\dfrac{L^2}{2}$ with probability greater than or equal to $1 - L^{-A}$. Moreover, denoting by $\phi$ the solution to the kinetic wave equation:
	\[
	\lim_{L \to +\infty} \sup_{t \in [0,1]} \sup_{k \in \mathbb{Z}_L^d} \left| \mathbb{E}\left[ \left| \widehat{u}(T_{\textup{kin}}\delta t, k) \right|^2 \right] - \phi(\delta t, k) \right| = 0.
	\]
\end{theorem}
In the theorem above, the wave kinetic equation is defined by:
\[
\begin{cases}
	\partial_t \phi = \mathcal{K}(\phi,\phi,\phi) & \text{sur } \mathbb{R}^+ \times \mathbb{R}^d \\
	\phi(0,\cdot) = \phi_{\text{in}}
\end{cases}
\]
with
\begin{equs}
	\mathcal{K}(\phi_1,\phi_2,\phi_3)(k) & =\int_{\Gamma(k)} 
	\left(\phi_1(k_1)\phi_2(k_2)\phi_3(k_3)-\phi_1(k)\phi_2(k_2)\phi_3(k_3) \right.
	\\ & \left. +\phi_1(k_1)\phi_2(k)\phi_3(k_3) -\phi_1(k_1)\phi_2(k_2)\phi_3(k)\right) dk_1 dk_2 d k_3
\end{equs}
and
\begin{equs}
\Gamma(k) & =\{(k_1,k_2,k_3)\in(\mathbb{R}^d)^3,k_1 - k_2 + k_3 - k = 0 \text{ and } \\ &
|k_1|_\beta^2 - |k_2|_\beta^2 + |k_3|_\beta^2 - |k|_\beta^2 = 0\}
\end{equs}

This remarkable result has been obtained  by pushing forward the analysis started in \cite{DH21} where the authors derive the wave kinetic equation arbitrarily close to the kinetic timescale based among other techniques on number theoretic results coming from \cite{BGHS21}. Let us summarise the main ideas of the proof of Theorem \ref{main_theorem_DH23} below.

 The study of this type of problem relies on the notion of decorated ternary signed trees: successive iterations of the Duhamel formula reveal a combinatorial structure in terms of iterated oscillatory integrals, where each term in the perturbative expansion corresponds to a decorated tree. The ansatz consists in writing:
$$
\mathbf{a}=\sum_{n=0}^N\mathbf{a}^{(n)}+\mathbf{R}_N
$$
where $\mathbf{a}^{(n)}$ are oscillatory iterated integral described by decorated trees with $n$ nodes, $\mathbf{R}_N$ is a reminder and $N\sim\log(L)$. If we denote by $ a_k $ the $k$-th Fourier coefficient of $ \mathbf{a} $, we take the correlations $\mathbb{E}[a_k(t)\overline{a_k(t)}]$ in the limit $L\to \infty$ with the scaling law $\lambda^2=L^{d-1}$. The (leading) combinatorial term can be written as a sum of particular paired trees, called couples (see \cite[Section 2.2]{DH23}). The main challenge is to classify and group these different couples in order to better understand their contributions. In \cite[Section 3.2]{DH23}, the couples are split in two sets:
\begin{itemize}
	\item the regular couples: they are defined recursively and they appear to have the same structure of the tree expansion of the wave kinetic equation described below (see \cite[Section 4]{DH23}). The regular couples are the leading term of the couple expansion.
	\item the irregular couples: they are not defined recursively, and so are more difficult to investigate. The aim is to show that the contribution of irregular couples vanish when taking the limit $L\to +\infty$.
\end{itemize}

All the combinatorial objects, trees, couples (see Definitions \eqref{def_tree_couple} and \eqref{decoration}) and their associated analytical interpretation are recalled in Section \ref{Sec::2} with some examples.  To treat the irregular couples, a new method is explained in  \cite[Section 9]{DH23}: a reduction algorithm. This algorithm is used to count the irregular couples with a rigidity theorem. It is one of the main tools  to derive the wave kinetic equation for $ \eqref{NLS_equa} $ in \cite{DH23} but also in \cite{DH26} for the propagation of chaos, in \cite{DH2301} for the full range of scaling laws and in \cite{DH2311} for a long-time derivation (as a black box in the last two works).
Let us mention that a simple version of this algorithm has been used in  \cite[Section 11]{BDNY24} for the analytical aspects of stochastic higher-order terms.
 The idea of \cite{DH23} is to rewrite couples in terms of molecules that are decorated graphs similar to Feynman diagrams. We recall the main definitions of these objects in the second part of Section \ref{Sec::2} (see Definitions \eqref{Molecule} and \eqref{molecule2}).
 Given a molecule, one applies a sophisticated cutting algorithm that produces good analytical estimates. 

The aim of this article is to characterise and to understand the main ideas behind the algorithm that allows to prove the rigidity theorem. Our main theorem is the following:

\begin{theorem} \label{main_theorem}
	The algorithm exposed in \cite[Section 9.3-9.4]{DH23} builds a spanning tree of a molecule $\mathbf{M}$ associated with an irregular couple in a Kruskal manner.
\end{theorem}

The proof of Theorem \ref{main_theorem} given in Section \ref{Sec::3:2} relies on an analysis of the various steps of the algorithm described in \cite{DH23} by keeping at each step a maximum number of edges removed. These edges are part of an acyclic graph in construction. This procedure is given in the proposition below

\begin{proposition}\label{acyclic_steps}
	We suppose given a step of the algorithm described in \cite[Section 9.3-9.4]{DH23} on a molecule $ \mathbf{M} $ and an acyclic graph $G=(V,E)$ where $V$ is a set that contains the atoms of $\mathbf{M}$ and $E$ is a  set of edges disjoint from the bonds of $\mathbf{M}$. Then, one can add a maximum number of edges to $G$ that are removed by the algorithm step we are considering such that the new graph obtained is still acyclic.
\end{proposition}

We perform the proof of the previous proposition in Section \ref{Sec::3:1} by reviewing all the steps of the algorithm given in \cite{DH23}.
Then the proof of the main theorem is based on a simple induction on the construction of the acyclic graph that becomes a spanning tree when the algorithm terminates. We say that such a spanning tree is constructed in a Kruskal manner because the algorithm is a loop over many steps ordered according to the arity of the nodes of the molecule. See Table \eqref{long_table} that provides the order of the steps with the number of cycles removed  and the exponent of the analytical bounds for each step. Then, this induces weights on the edges and one chooses to remove edges of minimal weights at each step. These weights are dynamical and evolve after several steps as they depend among other things on the arity of the nodes they are connected to. The arity of each node of the molecule decreases at some point moving along the various steps.

Previously in \cite{LS11} (proof of Theorem 5.1 and Proposition 5.2), the authors used spanning tree algorithm to deal with momentum graphs. This idea is reused in \cite{DST25} (see Appendix A). The algorithm described in these articles is Prim's algorithm and generates a spanning tree of the momentum graph starting from the root. This strategy may be non-optimal as starting from the root is arbitrary. The algorithm described in \cite{DH23} seems to be more powerful as it selects the edges according to a choice of weights. The main idea of this algorithm was reused in \cite{DHM25} to provide a long-time derivation of the Boltzmann equation. This solves a long-standing problem as the short-time derivation has been understood since the 70s in \cite{Land75}. The algorithm  was applied to particle collision history coming from \cite{BGSS22,BGSS23}. The authors used the cluster expansion method to derive diagrams with overlapping and colliding particles.  
One can expect to show that the main algorithm of \cite{DHM25} is also of Kruskal type.

We recall some basic facts about Prim's \cite{Pr57} and Kruskal's \cite{Kru56} algorithms (see also \cite[Chapter 21]{CLRS22} for an introduction to these algorithms).
Prim's algorithm is a greedy algorithm that finds a minimum spanning tree of a connected, weighted graph by starting from a vertex and repeatedly adding the smallest edge connecting the tree to a new vertex until all vertices are included.
\begin{figure}[H]
    \centering
    \begin{minipage}{0.44\textwidth}
        \centering
        \begin{tikzpicture}[scale=0.5, every node/.style={font=\small}, edge/.style={thick}, mst/.style={very thick, red}]
            \coordinate (A) at (0,0);
            \coordinate (B) at (2,1);
            \coordinate (C) at (4,0);
            \coordinate (D) at (1,-2);
            \coordinate (E) at (3,-2);

            \node at (2,-4) {Step 1};

            \draw[edge] (A) -- (B) node[midway, above] {2};
            \draw[edge] (B) -- (C) node[midway, above] {3};
            \draw[edge] (A) -- (D) node[midway, left] {6};
            \draw[edge] (B) -- (D) node[midway, left] {5};
            \draw[edge] (B) -- (E) node[midway, right] {4};
            \draw[edge] (C) -- (E) node[midway, right] {1};
            \draw[edge] (D) -- (E) node[midway, below] {7};

            \draw[mst] (A) -- (B);

            \node[above=0.5pt] at (A) {A};
            \node[above=0.5pt] at (B) {B};
            \node[above=0.5pt] at (C) {C};
            \node[below=0.5pt] at (D) {D};
            \node[below=0.5pt] at (E) {E};
        \end{tikzpicture}
    \end{minipage}
    \hspace{-3.4cm}
    \begin{minipage}{0.44\textwidth}
        \centering
        \begin{tikzpicture}[scale=0.5, every node/.style={font=\small}, edge/.style={thick}, mst/.style={very thick, red}]
            \coordinate (A) at (0,0);
            \coordinate (B) at (2,1);
            \coordinate (C) at (4,0);
            \coordinate (D) at (1,-2);
            \coordinate (E) at (3,-2);

            \node at (2,-4) {Step 2};

            \draw[edge] (A) -- (B) node[midway, above] {2};
            \draw[edge] (B) -- (C) node[midway, above] {3};
            \draw[edge] (A) -- (D) node[midway, left] {6};
            \draw[edge] (B) -- (D) node[midway, left] {5};
            \draw[edge] (B) -- (E) node[midway, right] {4};
            \draw[edge] (C) -- (E) node[midway, right] {1};
            \draw[edge] (D) -- (E) node[midway, below] {7};

            \draw[mst] (A) -- (B);
            \draw[mst] (B) -- (C);

            \node[above=0.5pt] at (A) {A};
            \node[above=0.5pt] at (B) {B};
            \node[above=0.5pt] at (C) {C};
            \node[below=0.5pt] at (D) {D};
            \node[below=0.5pt] at (E) {E};
        \end{tikzpicture}
        \end{minipage}
\hspace{-3.4cm}
        \begin{minipage}{0.44\textwidth}
        \centering
        \begin{tikzpicture}[scale=0.5, every node/.style={font=\small}, edge/.style={thick}, mst/.style={very thick, red}]
            \coordinate (A) at (0,0);
            \coordinate (B) at (2,1);
            \coordinate (C) at (4,0);
            \coordinate (D) at (1,-2);
            \coordinate (E) at (3,-2);

            \node at (2,-4) {Step 3};

            \draw[edge] (A) -- (B) node[midway, above] {2};
            \draw[edge] (B) -- (C) node[midway, above] {3};
            \draw[edge] (A) -- (D) node[midway, left] {6};
            \draw[edge] (B) -- (D) node[midway, left] {5};
            \draw[edge] (B) -- (E) node[midway, right] {4};
            \draw[edge] (C) -- (E) node[midway, right] {1};
            \draw[edge] (D) -- (E) node[midway, below] {7};

            \draw[mst] (A) -- (B);
            \draw[mst] (B) -- (C);
            \draw[mst] (C) -- (E);

            \node[above=0.5pt] at (A) {A};
            \node[above=0.5pt] at (B) {B};
            \node[above=0.5pt] at (C) {C};
            \node[below=0.5pt] at (D) {D};
            \node[below=0.5pt] at (E) {E};
        \end{tikzpicture}
    \end{minipage}
    \hspace{-3.4cm}
    \begin{minipage}{0.44\textwidth}
        \centering
        \begin{tikzpicture}[scale=0.5, every node/.style={font=\small}, edge/.style={thick}, mst/.style={very thick, red}]
            \coordinate (A) at (0,0);
            \coordinate (B) at (2,1);
            \coordinate (C) at (4,0);
            \coordinate (D) at (1,-2);
            \coordinate (E) at (3,-2);

            \node at (2,-4) {Step 4};

            \draw[edge] (A) -- (B) node[midway, above] {2};
            \draw[edge] (B) -- (C) node[midway, above] {3};
            \draw[edge] (A) -- (D) node[midway, left] {6};
            \draw[edge] (B) -- (D) node[midway, left] {5};
            \draw[edge] (B) -- (E) node[midway, right] {4};
            \draw[edge] (C) -- (E) node[midway, right] {1};
            \draw[edge] (D) -- (E) node[midway, below] {7};

            \draw[mst] (A) -- (B);
            \draw[mst] (B) -- (C);
            \draw[mst] (C) -- (E);
            \draw[mst] (B) -- (D);

            \node[above=0.5pt] at (A) {A};
            \node[above=0.5pt] at (B) {B};
            \node[above=0.5pt] at (C) {C};
            \node[below=0.5pt] at (D) {D};
            \node[below=0.5pt] at (E) {E};
        \end{tikzpicture}
    \end{minipage}
   \caption{Example of Prim minimum spanning tree rooted at $A$}
\end{figure}
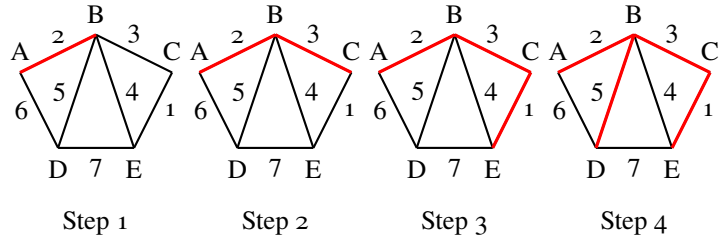

\medskip

Kruskal's algorithm is a greedy algorithm that finds a minimum spanning tree of a connected, weighted graph by repeatedly adding the smallest edge that doesnâ€™t form a cycle until all vertices are part of the spanning tree. We give a short example below.
\begin{figure}[H]
    \centering
    \begin{minipage}{0.44\textwidth}
        \centering
        \begin{tikzpicture}[scale=0.5, every node/.style={font=\small}, edge/.style={thick}, mst/.style={very thick, red}]
            \coordinate (A) at (0,0);
            \coordinate (B) at (2,1);
            \coordinate (C) at (4,0);
            \coordinate (D) at (1,-2);
            \coordinate (E) at (3,-2);

            \node at (2,-4) {Step 1};

            \draw[edge] (A) -- (B) node[midway, above] {2};
            \draw[edge] (B) -- (C) node[midway, above] {3};
            \draw[edge] (A) -- (D) node[midway, left] {6};
            \draw[edge] (B) -- (D) node[midway, left] {5};
            \draw[edge] (B) -- (E) node[midway, right] {4};
            \draw[edge] (C) -- (E) node[midway, right] {1};
            \draw[edge] (D) -- (E) node[midway, below] {7};
            \draw[mst] (C) -- (E);

            \node[above=0.5pt] at (A) {A};
            \node[above=0.5pt] at (B) {B};
            \node[above=0.5pt] at (C) {C};
            \node[below=0.5pt] at (D) {D};
            \node[below=0.5pt] at (E) {E};
        \end{tikzpicture}
    \end{minipage}
    \hspace{-3.4cm}
    \begin{minipage}{0.44\textwidth}
        \centering
        \begin{tikzpicture}[scale=0.5, every node/.style={font=\small}, edge/.style={thick}, mst/.style={very thick, red}]
            \coordinate (A) at (0,0);
            \coordinate (B) at (2,1);
            \coordinate (C) at (4,0);
            \coordinate (D) at (1,-2);
            \coordinate (E) at (3,-2);

            \node at (2,-4) {Step 2};

            \draw[edge] (A) -- (B) node[midway, above] {2};
            \draw[edge] (B) -- (C) node[midway, above] {3};
            \draw[edge] (A) -- (D) node[midway, left] {6};
            \draw[edge] (B) -- (D) node[midway, left] {5};
            \draw[edge] (B) -- (E) node[midway, right] {4};
            \draw[edge] (C) -- (E) node[midway, right] {1};
            \draw[edge] (D) -- (E) node[midway, below] {7};

            \draw[mst] (C) -- (E);
            \draw[mst] (A) -- (B);

            \node[above=0.5pt] at (A) {A};
            \node[above=0.5pt] at (B) {B};
            \node[above=0.5pt] at (C) {C};
            \node[below=0.5pt] at (D) {D};
            \node[below=0.5pt] at (E) {E};
        \end{tikzpicture}
    \end{minipage}
    \hspace{-3.4cm}
    \begin{minipage}{0.44\textwidth}
        \centering
        \begin{tikzpicture}[scale=0.5, every node/.style={font=\small}, edge/.style={thick}, mst/.style={very thick, red}]
            \coordinate (A) at (0,0);
            \coordinate (B) at (2,1);
            \coordinate (C) at (4,0);
            \coordinate (D) at (1,-2);
            \coordinate (E) at (3,-2);

            \node at (2,-4) {Step 3};

            \draw[edge] (A) -- (B) node[midway, above] {2};
            \draw[edge] (B) -- (C) node[midway, above] {3};
            \draw[edge] (A) -- (D) node[midway, left] {6};
            \draw[edge] (B) -- (D) node[midway, left] {5};
            \draw[edge] (B) -- (E) node[midway, right] {4};
            \draw[edge] (C) -- (E) node[midway, right] {1};
            \draw[edge] (D) -- (E) node[midway, below] {7};

            \draw[mst] (C) -- (E);
            \draw[mst] (A) -- (B);
            \draw[mst] (B) -- (C);

            \node[above=0.5pt] at (A) {A};
            \node[above=0.5pt] at (B) {B};
            \node[above=0.5pt] at (C) {C};
            \node[below=0.5pt] at (D) {D};
            \node[below=0.5pt] at (E) {E};
        \end{tikzpicture}
    \end{minipage}
    \hspace{-3.4cm}
    \begin{minipage}{0.44\textwidth}
        \centering
        \begin{tikzpicture}[scale=0.5, every node/.style={font=\small}, edge/.style={thick}, mst/.style={very thick, red}]
            \coordinate (A) at (0,0);
            \coordinate (B) at (2,1);
            \coordinate (C) at (4,0);
            \coordinate (D) at (1,-2);
            \coordinate (E) at (3,-2);

            \node at (2,-4) {Step 4};

            \draw[edge] (A) -- (B) node[midway, above] {2};
            \draw[edge] (B) -- (C) node[midway, above] {3};
            \draw[edge] (A) -- (D) node[midway, left] {6};
            \draw[edge] (B) -- (D) node[midway, left] {5};
            \draw[edge] (B) -- (E) node[midway, right] {4};
            \draw[edge] (C) -- (E) node[midway, right] {1};
            \draw[edge] (D) -- (E) node[midway, below] {7};

            \draw[mst] (C) -- (E);
            \draw[mst] (A) -- (B);
            \draw[mst] (B) -- (C);
            \draw[mst] (B) -- (D);

            \node[above=0.5pt] at (A) {A};
            \node[above=0.5pt] at (B) {B};
            \node[above=0.5pt] at (C) {C};
            \node[below=0.5pt] at (D) {D};
            \node[below=0.5pt] at (E) {E};
        \end{tikzpicture}
    \end{minipage}
    \caption{Example of Kruskal minimum spanning tree}
\end{figure}
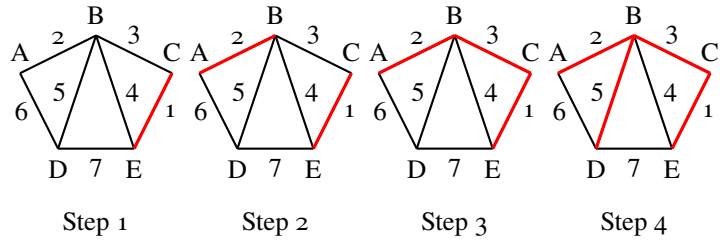
Both algorithms iteratively add edges to the spanning tree which is under construction. These edges are called safe edges in the sense that they do not create any cycles (see the generic spanning tree algorithm in \cite[Chapter 21]{CLRS22}) In this sense, Proposition \ref{acyclic_steps} guarantees that during the algorithm proposed in \cite{DH23}, one can add safe edges to the acyclic graph that will be in the end a spanning tree.

In the end, it is not so surprising that one can interpret the main algorithm of \cite{DH23} as a Kruskal type algorithm. Indeed, Kruskal spanning trees appear naturally in Quantum Field Theory when one wants to compute the BPHZ renormalisation (see \cite{BP57,H69,Z69}) on a Feynman diagram. They are used when one performs the multi-scale analysis forming the so-called Hepp sector that can be represented with the help of Gallavotti-Nicolo decorated trees (see \cite{GNI85I,GNI85II}). These techniques are surveyed in the book \cite{RiV91} and the notion of Kruskal spanning tree is mentioned in \cite{RW14}. They are also used to estimate the convergence radius of virial series, as in \cite{PY17}. Such a technique has been used in singular stochastic partial differential equations within the framework of Regularity Structures in \cite{HQ18} where the terminology of Kruskal algorithm is used. We terminate the survey of this literature by mentioning the works \cite{ESY07,ESY08} where the authors consider the case of the linear equation. They proceed with a careful analysis of the various Feynman diagrams and their nested structure. One may imagine reformulating their proof using a Kruskal type algorithm.

\subsection*{Acknowledgements}

{\small
	Y.B. and V.C. gratefully acknowledge funding support from the European Research Council (ERC) through the ERC Starting Grant Low Regularity Dynamics via Decorated Trees (LoRDeT), grant agreement No.\ 101075208. Views and opinions expressed are however those of the author(s) only and do not necessarily reflect those of the European Union or the European Research Council. Neither the European Union nor the granting authority can be held responsible for them.  Y.B. also thanks the "Institut des Hautes Études Scientifiques" (IHES) for a long research stay from 7th of January to 21st of March 2025, where the main idea for this work emerged.
	Y. B. gratefully acknowledges Thierry Bodineau for interesting discussions on the Boltzmann kinetic equation and for suggesting to apply for a long stay at IHES.
	 Y. B. thanks José Bruned for interesting discussions on mechanical systems where one uses similar graphical tools.
	Y. B. thanks Ismaël Bailleul for pointing out the use of spanning trees in Quantum Field Theory.
} 

\section{Trees, couples and molecules}

\label{Sec::2}

In this section, we recall the basic definitions from \cite{DH23} on the combinatorial objects that appear in the main algorithm studied in the next section.

\begin{definition}[Tree and couple] \label{def_tree_couple}
A tree is a signed ternary tree such that if a node has sign $\sigma$, 
then its children have signs, from left to right, $(\sigma,-\sigma,\sigma)$. We denote $\mathcal{N}(\mathcal{T}),\mathcal{L}(\mathcal{T})$ the set of nodes and the set of leaves of $\mathcal{T}$, and, if $\mathfrak{n}\in\mathcal{N}(\mathcal{T})$, $\sigma_\mathfrak{n}$ is the sign of $\mathfrak{n}$. We also denote $\mathfrak{r}(\mathcal{T})$ the root of $\mathcal{T}$.
A couple is a pair of trees $(\mathcal{T}^+,\mathcal{T}^-)$ such that $\sigma_{\mathfrak{r}(\mathcal{T}^\pm)}=\pm$, with a pairing $\mathcal{P}$ of the leaves $\mathcal{L}(\mathcal{T}^+)\cup\mathcal{L}(\mathcal{T}^-)$ in such a way that if $\{\mathfrak{l},\mathfrak{l}'\}\in\mathcal{P}$, then $\sigma_{\mathfrak{l}}=-\sigma_{\mathfrak{l}'}$.
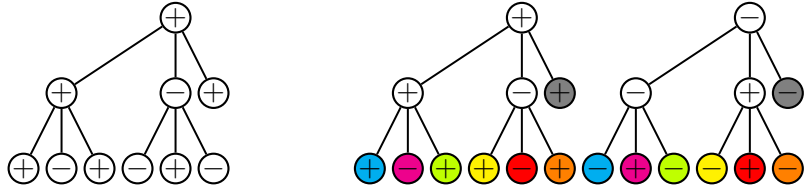
\begin{figure}[H]
\centering
\begin{tikzpicture}[scale=1, every node/.style={draw, circle, thick, minimum size=2mm, inner sep=0pt}]
	
	\node (1) at (0,0) {$+$};
	
	\node (11) [shift={(-1.5,-1)}] at (1) {$+$};
	\node (12) [shift={(0,-1)}] at (1) {$-$};
	\node (13) [shift={(0.5,-1)}] at (1) {$+$};
	\draw[-] (1) -- (11);
	\draw[-] (1) -- (12);
	\draw[-] (1) -- (13);
	
	\node (121) [shift={(-0.5,-1)}] at (12) {$-$};
	\node (122) [shift={(0,-1)}] at (12) {$+$};
	\node (123) [shift={(0.5,-1)}] at (12) {$-$};
	\draw[-] (12) -- (121);
	\draw[-] (12) -- (122);
	\draw[-] (12) -- (123);
	
	\node (111) [shift={(-0.5,-1)}] at (11) {$+$};
	\node (112) [shift={(0,-1)}] at (11) {$-$};
	\node (113) [shift={(0.5,-1)}] at (11) {$+$};
	\draw[-] (11) -- (111);
	\draw[-] (11) -- (112);
	\draw[-] (11) -- (113);
	
\end{tikzpicture} \qquad \qquad 
\begin{tikzpicture}[scale=1, every node/.style={draw, circle, thick, minimum size=2mm, inner sep=0pt}]

\node (A1) at (-1.5,0) {$+$};
\node (A11) [shift={(-1.5,-1)}] at (A1) {$+$};
\node (A12) [shift={(0,-1)}] at (A1) {$-$};
\node (A13) [shift={(0.5,-1)},fill=gray] at (A1) {$+$};
\draw[-] (A1) -- (A11);
\draw[-] (A1) -- (A12);
\draw[-] (A1) -- (A13);

\node (A111) [shift={(-0.5,-1)}, fill=cyan] at (A11) {$+$};
\node (A112) [shift={(0,-1)}, fill=magenta] at (A11) {$-$};
\node (A113) [shift={(0.5,-1)}, fill=lime] at (A11) {$+$};
\draw[-] (A11) -- (A111);
\draw[-] (A11) -- (A112);
\draw[-] (A11) -- (A113);

\node (A121) [shift={(-0.5,-1)}, fill=yellow] at (A12) {$+$};
\node (A122) [shift={(0,-1)}, fill=red] at (A12) {$-$};
\node (A123) [shift={(0.5,-1)}, fill=orange] at (A12) {$+$};
\draw[-] (A12) -- (A121);
\draw[-] (A12) -- (A122);
\draw[-] (A12) -- (A123);

\node (B1) at (1.5,0) {$-$};
\node (B11) [shift={(-1.5,-1)}] at (B1) {$-$};
\node (B12) [shift={(0,-1)}] at (B1) {$+$};
\node (B13) [shift={(0.5,-1)},fill=gray] at (B1) {$-$};
\draw[-] (B1) -- (B11);
\draw[-] (B1) -- (B12);
\draw[-] (B1) -- (B13);

\node (B111) [shift={(-0.5,-1)}, fill=cyan] at (B11) {$-$};
\node (B112) [shift={(0,-1)}, fill=magenta] at (B11) {$+$};
\node (B113) [shift={(0.5,-1)}, fill=lime] at (B11) {$-$};
\draw[-] (B11) -- (B111);
\draw[-] (B11) -- (B112);
\draw[-] (B11) -- (B113);

\node (B121) [shift={(-0.5,-1)}, fill=yellow] at (B12) {$-$};
\node (B122) [shift={(0,-1)}, fill=red] at (B12) {$+$};
\node (B123) [shift={(0.5,-1)}, fill=orange] at (B12) {$-$};
\draw[-] (B12) -- (B121);
\draw[-] (B12) -- (B122);
\draw[-] (B12) -- (B123);
\end{tikzpicture}
\caption{Examples of a tree and a couple. 
Signs are displayed inside each node of the tree.
	Paired leaves have the same color and they come with different sign.}
		\label{figure_tree}
\end{figure}
\end{definition}
\begin{definition}[Decoration]
	\label{decoration}
Let $\mathcal{T}$ a tree. A tuple $(k_{\mathfrak{n}})_{\mathfrak{n}\in\mathcal{N}(\mathcal{T})}\in (\mathbb{Z}_L^d)^{\mathcal{N}(\mathcal{T})}$ is a decoration of $\mathcal{T}$ if, for all $\mathfrak{n}\in\mathcal{N}(\mathcal{T})$ :
$$
\sigma_\mathfrak{n}k_\mathfrak{n}=\sigma_{\mathfrak{n}_1}k_{\mathfrak{n}_1}+\sigma_{\mathfrak{n}_2}k_{\mathfrak{n}_2}+\sigma_{\mathfrak{n}_3}k_{\mathfrak{n}_3}
$$
where $\mathfrak{n}_1,\mathfrak{n}_2,\mathfrak{n}_3$ are the children of $\mathfrak{n}$ from left to right.

Let $\mathcal{Q}=(\mathcal{T}^+,\mathcal{T}^-,\mathcal{P})$ a couple. A pair $(\mathcal{D}^+,\mathcal{D}^-)$ of decorations of $\mathcal{T}^\pm$ is a decoration of $\mathcal{Q}$, if, for all $\{\mathfrak{l},\mathfrak{l}'\}\in\mathcal{P}$, $k_{\mathfrak{l}}=k_{\mathfrak{l}'}$.
\end{definition}
In \cite[Section 2]{DH23}, the authors show that, denoting $b$ some remainder:
\begin{align*}
&\mathbb{E}\left[ a_k(t^+) \overline{a_k(t^-)} \right] \\
&= \sum_{(\mathcal{T}^+, \mathcal{T}^-) : \lvert\mathcal{N}(\mathcal{T}^\pm)\rvert \leqslant N} \left( \dfrac{\delta}{2 L^{d-1}} \right)^{\lvert\mathcal{N}(\mathcal{T}^+)\rvert +\lvert \mathcal{N}(\mathcal{T}^-)\rvert} \left( \prod_{\mathfrak{n} \in \mathcal{N}(\mathcal{T}^+) \cup \mathcal{N}(\mathcal{T}^-)} i \sigma_{\mathfrak{n}} \right) \\
&\quad \sum_{\mathcal{P} : \mathcal{Q}=(\mathcal{T}^+, \mathcal{T}^-, \mathcal{P}) \text{ couple}} \sum_{\mathcal{E}} \varepsilon(\mathcal{E}) \mathcal{B}_\mathcal{Q}(t^+, t^-, \delta L^2 \Omega[\mathcal{N}(\mathcal{T}^+) \cup \mathcal{N}(\mathcal{T}^-)]) \\ & \prod_{\mathfrak{l} \in \mathcal{L}^+} \phi_{\textup{in}}(k_\mathfrak{l}) + O(b)
\end{align*}
where $$\mathcal{B}_\mathcal{Q}(t^+, t^-, \alpha[\mathcal{N}(\mathcal{T}^+) \cup \mathcal{N}(\mathcal{T}^-)]) = \mathcal{A}_{\mathcal{T}^+}(t^+, \alpha[\mathcal{N}(\mathcal{T}^+)])\mathcal{A}_{\mathcal{T}^-}\left(t^-, \alpha[\mathcal{N}(\mathcal{T}^-)]\right),$$ the map $\varepsilon$ is a sign, $\mathcal{L}^+$ is the set of leaves of $\mathcal{Q}$ with $+$ sign, and, for all $\mathfrak{n}\in\mathcal{N}(\mathcal{T}^+)\cup\mathcal{N}(\mathcal{T}^-)$, $\Omega_\mathfrak{n}=|k_{\mathfrak{n}_1}|_\beta^2 - |k_{\mathfrak{n}_2}|_\beta^2 + |k_{\mathfrak{n}_3}|_\beta^2 - |k_{\mathfrak{n}}|_\beta^2$. We can give explicit formulae of $\mathcal{A},\mathcal{B}$ using iterated integrals:
\begin{align*}
\mathcal{A}_\mathcal{T}\left(t,\alpha[\mathcal{N}(\mathcal{T})]\right)
&=
\int_{\mathcal{D}\left(\mathcal{N}(\mathcal{T}),t\right)}
\prod_{\mathfrak{n}\in\mathcal{N}(\mathcal{T})}
e^{i\pi\sigma_{\mathfrak{n}}\alpha_\mathfrak{n}t_\mathfrak{n}}
\, dt_\mathfrak{n},
\\
\mathcal{B}_{\mathcal{Q}}(t^+,t^-,\alpha[\mathcal{N}(\mathcal{T}^+)\cup\mathcal{N}(\mathcal{T}^-)])
&=
\int_{\mathcal{E}\left(\mathcal{N}(\mathcal{T})\cup \mathcal{N}(\mathcal{T}^-),t^+,t^-\right)} \\ & 
\prod_{\mathfrak{n}\in\mathcal{N}(\mathcal{T}^+)\cup\mathcal{N}(\mathcal{T}^-)}
e^{i\pi\sigma_{\mathfrak{n}}\alpha_\mathfrak{n}t_\mathfrak{n}}
\, dt_\mathfrak{n}
\end{align*}
with:
$$
\begin{aligned}
&\mathcal{D}\left(\mathcal{N}(\mathcal{T}),t\right) = \left\{ t[\mathcal{N}(\mathcal{T})], \; 0 < t_{\mathfrak{n}'} < t_\mathfrak{n} < t \text{ if } \mathfrak{n}' \text{ is a child of } \mathfrak{n} \right\}\\
&\mathcal{E}\left(\mathcal{N}(\mathcal{T})\cup \mathcal{N}(\mathcal{T}^-),t^+,t^-\right)\\& = \left\{t[\mathcal{N}(\mathcal{T})\cup \mathcal{N}(\mathcal{T}^-)], \; 0 < t_{\mathfrak{n}'} < t_\mathfrak{n} \text{ si } \mathfrak{n}' \text{ is a child of } \mathfrak{n}, t_\mathfrak{n} < t^\pm \text{ if } \mathfrak{n} \in \mathcal{N}\left(\mathcal{T}^\pm\right) \right\}.
\end{aligned}
$$
If $\mathcal{T}$ is the tree given in Figure \ref{figure_tree}, we have:
$$
\begin{aligned}
\mathcal{A}_\mathcal{T}(t,\alpha[\mathcal{N}(\mathcal{T})]) =\int_0^t dt_1\;e^{i\pi\sigma_1\alpha_1t_1}\frac{e^{i\pi\sigma_{11}\alpha_{11}t_1}-1}{i\pi\sigma_{11}\alpha_{11}}
\frac{e^{i\pi\sigma_{12}\alpha_{12}t_1}-1}{i\pi\sigma_{12}\alpha_{12}}
\end{aligned}
$$
where, if $i$ is the label of a node, $i1,i2,i3$ are the labels of its children from left to right.
\begin{definition}[Molecule]
	\label{Molecule}
	A molecule is a directed graph whose vertices are called \emph{atoms} and whose edges are called \emph{bonds}. We write $l\sim v$ if $l$ is a bond is arriving at $v$. We further assume that:
	\begin{itemize}
		\item Each atom has at most $2$ incoming bonds and $2$ outgoing bonds.
		\item There is no connected component such that each of its atoms has $4$ bonds.
	\end{itemize}
	We denote $\chi$ the Euler characteristic of the molecule (i.e. the number of independent cycles).
We also define the following two molecular subgroups:
	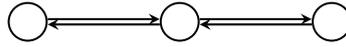
\begin{figure}[H]
		\begin{center}
			\begin{tikzpicture}[>=stealth, scale=1]
\tikzset{
	nodec/.style={circle, draw, minimum size=5mm, inner sep=0pt},
	arrow/.style={->, thick},
	dashedlink/.style={dashed, thick}
}

\node[nodec] (A) at (0,0) {};
\node[nodec] (B) at (2,0) {};
\node[nodec] (C) at (4,0) {};

\draw[arrow] ([yshift=1pt]A.east) -- ([yshift=1pt]B.west);
\draw[arrow] ([yshift=-1pt]B.west) -- ([yshift=-1pt]A.east);

\draw[arrow] ([yshift=1pt]B.east) -- ([yshift=1pt]C.west);
\draw[arrow] ([yshift=-1pt]C.west) -- ([yshift=-1pt]B.east);

\end{tikzpicture}
		\end{center}
		\caption{Type I molecular chain}
	\end{figure}
	\begin{figure}[H]
		\begin{center}
			\begin{tikzpicture}[>=stealth, scale=1]

\tikzset{
	nodec/.style={circle, draw, minimum size=5mm, inner sep=0pt},
	arrow/.style={->, thick},
	vertlink/.style={thick},
	dashedlink/.style={dashed, thick}
}

\def\xsep{2}
\def\ysep{1}

\foreach \i in {1,...,4} {
	\node[nodec] (T\i) at ({(\i-1)*\xsep}, \ysep) {};
}

\foreach \i in {1,...,4} {
	\node[nodec] (B\i) at ({(\i-1)*\xsep}, 0) {};
}

\foreach \i in {1,...,4} {
	\draw[vertlink] ([xshift=-1pt]T\i.south) -- ([xshift=-1pt]B\i.north);
	\draw[vertlink] ([xshift= 1pt]T\i.south) -- ([xshift= 1pt]B\i.north);
}

\draw[arrow] (T1) -- (T2);
\draw[arrow] (T3) -- (T2);
\draw[arrow] (T3) -- (T4);

\draw[arrow] (B2) -- (B1);
\draw[arrow] (B2) -- (B3);
\draw[arrow] (B4) -- (B3);

\draw[dashedlink] (-0.5, \ysep) -- (T1);
\draw[dashedlink] (-0.5, 0) -- (B1);

\draw[arrow] (B4) -- ({3*\xsep+0.5}, 0);
\draw[arrow] ({3*\xsep+0.5}, \ysep) -- (T4);

\end{tikzpicture}
		\end{center}
		\caption{Type II molecular chain}
	\end{figure}
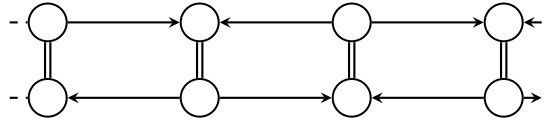
\end{definition}

\begin{proposition}
	Between two atoms, there is at most one triple bond. If there are $n$ atoms, then there are at most $2n-1$ bonds. In the case of equality, we speak of a base molecule.
	A base molecule is connected. It has two atoms of degree $3$ or one atom of degree $2$, and the others are of degree $4$.
\end{proposition}
\begin{proof}
	See  \cite[Proposition 9.2]{DH23}.
\end{proof}
\begin{definition}[Base molecule associated with a couple] \label{molecule2}
	Let $\mathcal{Q}$ be a nontrivial couple. The molecule associated with $\mathcal{Q}$ is defined by:
	\begin{itemize}
		\item Atom: quadruple $(\mathfrak{n},\mathfrak{n}_1,\mathfrak{n}_2,\mathfrak{n}_3)$ where $\mathfrak{n}$ is a parent and $\mathfrak{n}_1,\mathfrak{n}_2,\mathfrak{n}_3$ are its children from left to right.
		\item Bond: two atoms are connected by a bond if
		\begin{center}
			(PC) -- a node is a parent in one atom $v_1$ and a child in the other $v_2$. If $\mathfrak{n}$ is a parent in $v_1$, a child in $v_2$, and $\sigma_\mathfrak{n}=+$, then the bond goes from $v_1$ to $v_2$. The bond goes from $v_2$ to $v_1$ otherwise.
		\end{center}
		or
		\begin{center}
			(LP) -- one leaf of each atom forms a pair of $\mathcal{Q}$. The bond goes from the atom containing the leaf with sign $-$ to the atom containing the leaf with sign $+$.
		\end{center}
	\end{itemize}
	The molecule thus formed is a base molecule.
\end{definition}

\begin{proof}
	See  \cite[Proposition 9.4]{DH23}.
\end{proof}
Below, we provide a couple and its associated molecule:
\begin{figure}[H]
\centering
\begin{tikzpicture}[scale=1, every node/.style={draw, circle, thick, minimum size=2mm, inner sep=0pt}]
	
	\node (A1) at (-1.5,0) {$+$};
	\node (A11) [shift={(-1.5,-1)}] at (A1) {$+$};
	\node (A12) [shift={(0,-1)}] at (A1) {$-$};
	\node (A13) [shift={(0.5,-1)},fill=gray] at (A1) {$+$};
	\draw[-] (A1) -- (A11);
	\draw[-] (A1) -- (A12);
	\draw[-] (A1) -- (A13);
	
	\node (A111) [shift={(-0.5,-1)}, fill=cyan] at (A11) {$+$};
	\node (A112) [shift={(0,-1)}, fill=magenta] at (A11) {$-$};
	\node (A113) [shift={(0.5,-1)}, fill=lime] at (A11) {$+$};
	\draw[-] (A11) -- (A111);
	\draw[-] (A11) -- (A112);
	\draw[-] (A11) -- (A113);
	
	\node (A121) [shift={(-0.5,-1)}, fill=yellow] at (A12) {$+$};
	\node (A122) [shift={(0,-1)}, fill=red] at (A12) {$-$};
	\node (A123) [shift={(0.5,-1)}, fill=orange] at (A12) {$+$};
	\draw[-] (A12) -- (A121);
	\draw[-] (A12) -- (A122);
	\draw[-] (A12) -- (A123);
	
	\node (B1) at (1.5,0) {$-$};
	\node (B11) [shift={(-1.5,-1)}] at (B1) {$-$};
	\node (B12) [shift={(0,-1)}] at (B1) {$+$};
	\node (B13) [shift={(0.5,-1)},fill=gray] at (B1) {$-$};
	\draw[-] (B1) -- (B11);
	\draw[-] (B1) -- (B12);
	\draw[-] (B1) -- (B13);
	
	\node (B111) [shift={(-0.5,-1)}, fill=cyan] at (B11) {$-$};
	\node (B112) [shift={(0,-1)}, fill=magenta] at (B11) {$+$};
	\node (B113) [shift={(0.5,-1)}, fill=lime] at (B11) {$-$};
	\draw[-] (B11) -- (B111);
	\draw[-] (B11) -- (B112);
	\draw[-] (B11) -- (B113);
	
	\node (B121) [shift={(-0.5,-1)}, fill=yellow] at (B12) {$-$};
	\node (B122) [shift={(0,-1)}, fill=red] at (B12) {$+$};
	\node (B123) [shift={(0.5,-1)}, fill=orange] at (B12) {$-$};
	\draw[-] (B12) -- (B121);
	\draw[-] (B12) -- (B122);
	\draw[-] (B12) -- (B123);
	
\end{tikzpicture}
\qquad \qquad 
\begin{tikzpicture}[scale=1, every node/.style={draw, circle, thick, minimum size=3mm,inner sep=0pt}]
\node (T1) at (0,0) {};
\node (T2) at (2,0) {};
\node (T4) at (1,0) {};
\node (B1) at (0,-1) {};
\node (B2) at (2,-1) {};
\node (B4) at (1,-1) {};

\draw[<-, cyan] (T1) to[out=-120, in=120] (B1);
\draw[->, magenta] (T1) -- (B1);
\draw[<-, lime] (T1) to[out=-60, in=60] (B1);

\draw[<-, yellow] (T2) to[out=-120, in=120] (B2);
\draw[->, red] (T2) -- (B2);
\draw[<-, orange] (T2) to[out=-60, in=60] (B2);

\draw[->, gray] (B4) to (T4);

\draw[->] (T1) -- (T4);
\draw[<-] (T2) -- (T4);
\draw[<-] (B1) -- (B4);
\draw[->] (B2) -- (B4);

\end{tikzpicture}
\caption{Couple and its associated base molecule. The colors of the pairings of the couple correspond to the colors of the edges in the molecule.}
\end{figure}
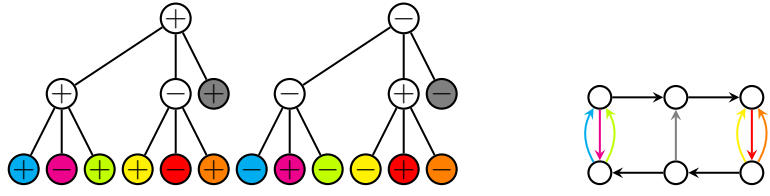
\begin{definition}[Admissible decorations]
	Let $\mathbf{M}$ be a molecule and $S$ a set of atoms, called degenerated atoms. We fix:
	\begin{itemize}
		\item For all bond $l$ of $\mathbf{M}$, $a_l\in\mathbb{Z}_L^d$
		\item For every non-isolated atom $v$ of $\mathbf{M}$, $c_v\in\mathbb{Z}_L^d$, with $c_v=0$ if $v$ has degree $4$
		\item For every non-isolated atom $v$ of $\mathbf{M}$, $\Gamma_v\in\mathbb{R}$
		\item For every $v\in S$ of degree at most $3$, $f_v\in\mathbb{Z}_L^d$
	\end{itemize}
	We denote by $\mathfrak{D}(\mathbf{M},S,a,c,\Gamma,f)$ the set of decorations $(k_l)_{l\in \mathbf{M}}$ such that, for all bond $l$ of $\mathbf{M}$:
 \begin{equs}
 	\begin{aligned}
 		k_l\in\mathbb{Z}_L^d, \quad & \lvert k_l-a_l\rvert\leqslant 1, \quad \sum_{v:l\sim v}\sigma(v,l)k_l=c_v,
 		\\ & \displaystyle\left\lvert \sum_{l:l\sim v}\sigma(v,l)\lvert k_l\rvert_\beta^2-\Gamma_v\right\rvert\leqslant \dfrac{1}{\delta L^2},
 		\end{aligned}
 		\end{equs}
	 where $\sigma(l,v)$ equals $1$ if $l$ leaves $v$, and $-1$ otherwise. Otherwise
		 if $v\in S$, the $k_l$ for $l\sim v$ are equal if $v$ is non-isolated, and equal to $f_v$ if $d(v)<4$.
		 If $v\notin S$ is non-isolated and $l_1,l_2\sim v$ are in opposite directions, then $k_{\mathfrak{l}_1}\neq k_{\mathfrak{l}_2}$.
	
	We write more simply $\mathfrak{D}(\mathbf{M})$. If additional conditions $\mathrm{Ext}$ of the form $k_{l_1}-k_{l_2}\in E$ with $E\subset\mathbb{Z}_L^d$ are added, we write $\mathfrak{D}(\mathbf{M},\mathrm{Ext})$.
\end{definition}

\begin{remark}
	A decoration of a couple $\mathcal{Q}$ induces a decoration of the associated base molecule $\mathbf{M}$, see \cite[Remark 9.10]{DH23}.
\end{remark}

\section{Rigidity theorem and algorithm}
We start the section by recalling the Rigidity Theorem, which is the main analytical result given in \cite{DH23} that provides analytical bounds on base molecules.
\begin{theorem}[Rigidity]
	Let $\mathbf{M}$ be a base molecule with $n$ atoms, where $1\leqslant n \leqslant (\log L)^3$, with no triple bond. Then $\mathfrak{D}(\mathbf{M})$ is the union of at most $C^n$ subsets of the form $\mathfrak{D}(\mathbf{M}, \mathrm{Ext})$, and there exist $r\in\{ 1,...,n\}$ and a collection of at most $Cr$ molecular chains of type I or II in $\mathbf{M}$ such that:
	\begin{itemize}
		\item The number of atoms in each molecular chain is bounded above by $Cr$
		\item For any molecular chain of type II, for every bond $l_1,l_2$ in this molecular chain facing each other, $\mathrm{Ext}$ includes the condition $k_{l_1}=k_{l_2}$
		\item One has
		$$
		\sup\#\mathfrak{D}(\mathbf{M},\mathrm{Ext})\leqslant (C^+)^n\dfrac{1}{\delta^{(n+m)/2}}L^{(d-1)n-2\nu r}
		$$
		where $\mathbf{M}$ is the number of atoms in the molecular chains of type I, $\nu=\dfrac{1}{100d}$, and $C^+$ is some constant depending on $(d,\beta,\varphi_{\text{in}})$.
	\end{itemize}
\end{theorem}
The proof performed in \cite[Section 9.5]{DH23} relies on an algorithm that modifies the molecule until all atoms are removed. A \emph{step} is a modification of the molecule and the set of conditions $\mathrm{Ext}$. A \emph{path} is a sequence of steps that ends at the previous state. A \emph{checkpoint} is a moment after a step that can continue along two different paths. Each path contains at most $Cn$ steps, and there are at most $C^n$ paths, for some constant $C\in\mathbb{R}_+^*$. We also introduce two quantities $(\gamma,\kappa)$ whose variations are defined at each step. Given a time marker, we consider the paths that coincide up to that marker and yield different sets $\mathrm{Ext}$. We denote by $\Upsilon$ the collection of $\mathrm{Ext}$ thus obtained. At each instant:
\begin{itemize}
	\item $\mathbf{M}$ is a molecule.
	\item The vectors $k[\mathbf{M}]$ satisfy one of the conditions $\mathrm{Ext}\in\Upsilon$.
	\item If $\mathbf{M}$ is the union of sub-molecules $\mathbf{M}_j$, then $\mathrm{Ext}$ is the union of the $\mathrm{Ext}_j$
	\item One has
	$$
	\sup\#\mathfrak{D}(\mathbf{M},\mathrm{Ext})\leqslant (C^+)^{n_0}\delta^{-\kappa}L^{(d-1)\gamma}
	$$
	with $n_0$ the number of steps remaining along the path
\end{itemize}
The last point is the most difficult to verify, and it is proved in \cite[Section 9.3]{DH23}:
$$
\sup\#\mathfrak{D}(\mathbf{M}_\text{pre},\mathrm{Ext}_{\text{pre}})\leqslant C^+\delta^{\Delta\kappa}L^{-(d-1)\Delta\gamma}\sup\#\mathfrak{D}(\mathbf{M}_\text{pos},\mathrm{Ext}_{\text{pos}})
$$
where $ \mathbf{M}_\text{pre},\mathrm{Ext}_{\text{pre}} $ (resp. $ \mathbf{M}_\text{pos},\mathrm{Ext}_{\text{pos}} $) correspond to the values of $ \mathbf{M},\mathrm{Ext} $ before (resp. after) a given step of the algorithm.

It appears that this algorithm is similar to a Kruskal algorithm for constructing a spanning tree of the molecule. The quantities $ \Delta \kappa $ and $\Delta \gamma$ correspond to the variation of $\kappa$ and $\gamma$ during one step.

The remainder of this section is devoted to the proof of Theorem \ref{main_theorem} by reviewing each step of the algorithm in order to build the requested spanning tree. We start by showing at each step of the algorithm that we are growing a spanning tree.
\subsection{Steps and proof of Proposition \ref{acyclic_steps}}
\label{Sec::3:1}
Regarding the steps of the algorithm, we always have:
$$
\Delta\gamma=\Delta\chi \text{ or } \Delta\gamma\geqslant \Delta\chi+\dfrac{1}{6(d-1)}
$$
In the first case, the step is said to be \emph{normal}. In the second case, the step is said to be \emph{good}.
One notices that for a \emph{normal} step, one reduces the number of cycles in the molecule, as $ \Delta \gamma $ is equal to the variation of cycles, which is negative as we are removing nodes and edges at each step.
The \emph{good} steps allow one to get a better bound. 
In the next proof, we review all the steps of this algorithm and we show that one can select edges such that they will be part of a spanning tree. We omit the description of $\mathrm{Ext}$ and its variation, as it is not needed for constructing the spanning tree. As for $\gamma$ and $\kappa$, as they appear in the analytical bounds, we have included them in Table \ref{long_table}. Indeed, $\gamma$ is important for determining if a step is good or normal. They can help to understand the order, but there is no general rule, as a normal step can be performed in very specific situations before a good step.

\begin{proof}[of Proposition \ref{acyclic_steps}]

We describe the different possible steps during the algorithm and we color in red the edges which are added to the acyclic graph $G$ that will be a spanning tree when one terminates the algorithm described in the next section. The connected components are dark-dashed boxes. In all the steps, the vertices of $G$ contain the atoms of $\mathbf{M}$, and $E_{\text{pos}}=E_{\text{pre}}\cup \Delta E$ where $E_{\text{pos}},E_{\text{pre}}$ are the sets of edges of $G$ after and before the step, and $\Delta E$ is the set of edges added to $G$ during the step.
These edges are a subset of the edges removed at each step. 

\textbf{Step DA}: We remove a non-isolated degenerate atom $v$ and all its bonds. We add the bonds to $G$ that do not create a cycle.

 In the following steps, we assume there is no degenerate atom.

\vspace{0.5em}
\textbf{Step TB}: Two atoms $v_1,v_2$ are connected by a triple bond. There are two possible sub-steps during which we remove from the molecule the two atoms and all the bonds connected to them:
\begin{itemize}
\item TB1 if $d(v_1)=d(v_2)=3$
\begin{center}
\begin{tikzpicture}[scale=1]

\tikzset{
  nodec/.style={circle, draw=black, minimum size=5mm, inner sep=0pt},
  edge/.style={very thick},
  rededge/.style={very thick, red}
}

\node[nodec] (A) at (0,0) {\scriptsize $v_1$};
\node[nodec] (B) at (3,0) {\scriptsize $v_2$};

\draw[rededge] ([yshift=3pt]A.east) -- ([yshift=3pt]B.west);
\draw[rededge] (A.east) -- (B.west);
\draw[rededge] ([yshift=-3pt]A.east) -- ([yshift=-3pt]B.west);

\end{tikzpicture}
\end{center}
Above, one faces the scenario that adding the edge $(v_1,v_2)$ does not create a cycle in the graph $G$. Otherwise, this edge is not added.
    \item TB2 if $d(v_1)=3,d(v_2)=4$
\begin{center}
\begin{tikzpicture}[scale=1]
\tikzset{
  nodec/.style={circle, draw=black, minimum size=5mm, inner sep=0pt},
  edge/.style={very thick},
  rededge/.style={very thick, red}
}

\node[nodec] (A) at (0,0) {\scriptsize $v_1$};
\node[nodec] (B) at (3,0) {\scriptsize $v_2$};

\draw[rededge] ([yshift=3pt]A.east) -- ([yshift=3pt]B.west);
\draw[rededge] (A.east) -- (B.west);
\draw[rededge] ([yshift=-3pt]A.east) -- ([yshift=-3pt]B.west);

\draw[rededge] (B.east) -- ++(1.5,0);
\end{tikzpicture}

\end{center}
Above, the red edges are added to $G$. Indeed, $v_2$ is of maximal arity that is $4$ and therefore it is an isolated node in $G$.
\end{itemize}

\vspace{0.5em}

\textbf{Step BR}: Two atoms $v_1,v_2$ are connected by one bond (called a "bridge") such that, if we remove this bond from the molecule, the number of connected components of the molecule increases by $1$. We assume that the molecule no longer has any triple bond. We remove the bond $(v_1,v_2)$ from the molecule, and we add it to the acyclic graph $G$.
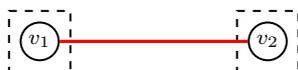
\begin{center}
\begin{tikzpicture}[scale=1]

\tikzset{
  nodec/.style={circle, draw=black, minimum size=5mm, inner sep=0pt},
  rededge/.style={very thick, red},
  box/.style={draw, dashed, black, inner sep=4pt}
}

\node[nodec] (A) at (0,0) {\scriptsize $v_1$};
\node[nodec] (B) at (3,0) {\scriptsize $v_2$};

\node[box, fit=(A)] {};
\node[box, fit=(B)] {};

\draw[rededge] (A) -- (B);

\end{tikzpicture}

\end{center}

\textbf{Step 3S3}: Two degree $3$ atoms $v_1,v_2$ are connected by a single bond. For details, see 9.3.4 in \cite{DH23}.
\begin{itemize}
\item Step (3S3-1): We remove from the molecule the atoms $\{v_1, v_2\}$ and all $5$ bonds connecting them.

    \item Step (3S3-2G): We remove from the molecule $\{v_1, v_2\}$ and all bonds connecting them.

    \item Step (3S3-3G): We remove from the molecule $\{v_1, v_2\}$ and all bonds connecting them, we add a new bond $\ell_6$ between $v_3$ and $v_5$, which goes from $v_3$ to $v_5$ if $\ell_2$ goes from $v_3$ to $v_1$, and vice versa.

    \item Steps (3S3-4G)--(3S3-5G): We remove from the molecule $v_1, v_2$ and all bonds connecting them.
\end{itemize}

We add to the acyclic graph $G$ the maximum number of bonds linked to $v_1$ or $v_2$ such that there is no cycle created when they are added to the acyclic graph $G$. One can observe that at this point the choice of these edges is not unique. We suppose that an arbitrary order is fixed on the edges. Below, we provide different scenarios:
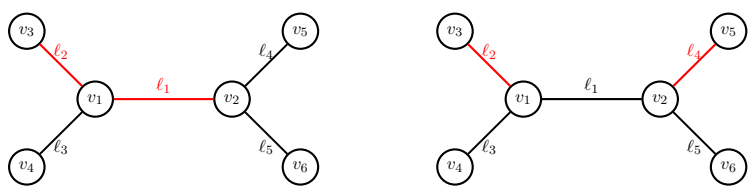
\begin{equation*}
\begin{tikzpicture}[scale=0.6, every node/.style={scale=0.6}]
    \node[draw=black, circle, thick, minimum size=5mm] (v1) at (0,0) {$v_1$};
    \node[draw=black, circle, thick, minimum size=5mm] (v2) at (3,0) {$v_2$};
    \node[draw=black, circle, thick, minimum size=5mm] (v3) at (-1.5,1.5) {$v_3$};
    \node[draw, circle, thick, minimum size=5mm] (v4) at (-1.5,-1.5) {$v_4$};
    \node[draw, circle, thick, minimum size=5mm] (v5) at (4.5,1.5) {$v_5$};
    \node[draw, circle, thick, minimum size=5mm] (v6) at (4.5,-1.5) {$v_6$};

    \draw[red, thick] (v1) -- node[above]{$\ell_1$} (v2);
    \draw[red, thick] (v1) -- node[above]{$\ell_2$} (v3);
    \draw (v1) -- node[below]{$\ell_3$} (v4);
    \draw (v2) -- node[above]{$\ell_4$} (v5);
    \draw (v2) -- node[below]{$\ell_5$} (v6);
\end{tikzpicture} \qquad  \qquad
	\begin{tikzpicture}[scale=0.6, every node/.style={scale=0.6}]
		\node[draw=black, circle, thick, minimum size=5mm] (v1) at (0,0) {$v_1$};
		\node[draw=black, circle, thick, minimum size=5mm] (v2) at (3,0) {$v_2$};
		\node[draw=black, circle, thick, minimum size=5mm] (v3) at (-1.5,1.5) {$v_3$};
		\node[draw, circle, thick, minimum size=5mm] (v4) at (-1.5,-1.5) {$v_4$};
		\node[draw=black, circle, thick, minimum size=5mm] (v5) at (4.5,1.5) {$v_5$};
		\node[draw, circle, thick, minimum size=5mm] (v6) at (4.5,-1.5) {$v_6$};
		
		\draw[thick] (v1) -- node[above]{$\ell_1$} (v2); 
		\draw[red, thick] (v1) -- node[above]{$\ell_2$} (v3);
		\draw (v1) -- node[below]{$\ell_3$} (v4);
		\draw[red, thick] (v2) -- node[above]{$\ell_4$} (v5);
		\draw (v2) -- node[below]{$\ell_5$} (v6);
	\end{tikzpicture}
\end{equation*}
On the left, we have added $\ell_1$ and $\ell_2$ to the acyclic graph $G$, and the other edges cannot be added due to a cycle. On the right, we have added $\ell_2$ and $\ell_4$, and we are in the case where the other edges produce a cycle.
In step (3S3-3G), one adds a new edge $\ell_6$ between nodes that were originally connected to $v_1$ and $v_2$. This new edge will not be added in the following steps, as it will create a cycle.
\vspace{0.5cm}

\textbf{Step 3D3}: Two degree $3$ atoms $v_1,v_2$ are connected by a double bond. For details, see 9.3.5 in \cite{DH23}.
\begin{itemize}
  \item Step (3D3-1): We remove from the molecule the atoms $\{v_1, v_2\}$ and all four bonds connected to them.

  \item Step (3D3-2G): We remove from the molecule $\{v_1, v_2\}$ and all bonds connected to them.

  \item Step (3D3-3G): We remove from the molecule $\{v_1, v_2\}$ and all bonds connected to them, 
  but we add to the molecule a new bond $\ell_5$ between $v_3$ and $v_4$ (not shown in Figure 29), 
  which goes from $v_4$ to $v_3$ if $\ell_3$ goes from $v_1$ to $v_3$, and vice versa.

  \item Steps (3D3-4G)--(3D3-5G): We remove from the molecule $v_1$ and $v_2$, as well as all bonds connected to them.

  \item Step (3D3-6G): We remove from the molecule the atoms $\{v_1, v_2,v_3, v_4\}$ and all nine bonds connected to them.
\end{itemize}
For all the steps except 3D3-6G, we add to the acyclic graph $G$ the maximum number of bonds linked to $v_1$ or $v_2$ such that there is no cycle created when they are added to the acyclic graph $G$. For 3D3-6G, we proceed as follows:
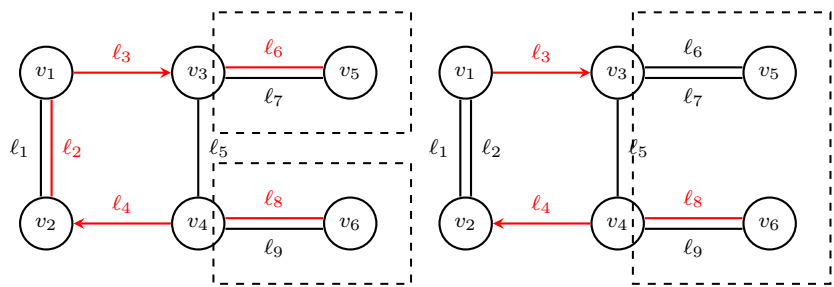
\begin{figure}[H]
\begin{center}
\begin{tikzpicture}[scale=1]
    \node[circle, draw=black, thick, minimum size=5mm] (v1) at (0,2) {\scriptsize $v_1$};
    \node[circle, draw=black, thick, minimum size=5mm] (v2) at (0,0) {\scriptsize $v_2$};
    \node[circle, draw=black, thick, minimum size=5mm] (v3) at (2,2) {\scriptsize $v_3$};
    \node[circle, draw=black, thick, minimum size=5mm] (v4) at (2,0) {\scriptsize $v_4$};
    \node[circle, draw=black, thick, minimum size=5mm] (v5) at (4,2) {\scriptsize $v_5$};
    \node[circle, draw=black, thick, minimum size=5mm] (v6) at (4,0) {\scriptsize $v_6$};

    \draw[thick] ([xshift=-2pt]v1.south) -- node[midway,left] {\scriptsize $\ell_1$} ([xshift=-2pt]v2.north);
    \draw[thick, red] ([xshift=2pt]v1.south) -- node[midway,right] {\scriptsize$\ell_2$} ([xshift=2pt]v2.north);

    \draw[->, red, thick] (v1) -- node[midway,above] {\scriptsize $\ell_3$} (v3);
    \draw[->, red, thick] (v4) -- node[midway,above] {\scriptsize $\ell_4$} (v2);
    \draw[thick] (v3) -- node[midway,right] {\scriptsize $\ell_5$} (v4);

    \draw[thick, red] ([yshift=2pt]v3.east) -- node[midway,above] {\scriptsize $\ell_6$} ([yshift=2pt]v5.west);
    \draw[thick] ([yshift=-2pt]v3.east) -- node[midway,below] {\scriptsize $\ell_7$} ([yshift=-2pt]v5.west);

    \draw[thick, red] ([yshift=2pt]v4.east) -- node[midway,above] {\scriptsize $\ell_8$} ([yshift=2pt]v6.west);
    \draw[thick] ([yshift=-2pt]v4.east) -- node[midway,below] {\scriptsize $\ell_9$} ([yshift=-2pt]v6.west);

    \draw[dashed] (2.2,2.8) rectangle (4.8,1.2);
    \draw[dashed] (2.2,0.8) rectangle (4.8,-0.8);
\end{tikzpicture}
\begin{tikzpicture}[scale=1]
    \node[circle, draw=black, thick, minimum size=6mm] (v1) at (0,2) {\scriptsize $v_1$};
    \node[circle, draw=black, thick, minimum size=6mm] (v2) at (0,0) {\scriptsize $v_2$};
    \node[circle, draw=black, thick, minimum size=6mm] (v3) at (2,2) {\scriptsize $v_3$};
    \node[circle, draw=black, thick, minimum size=6mm] (v4) at (2,0) {\scriptsize $v_4$};
    \node[circle, draw=black, thick, minimum size=6mm] (v5) at (4,2) {\scriptsize $v_5$};
    \node[circle, draw=black, thick, minimum size=6mm] (v6) at (4,0) {\scriptsize $v_6$};

    \draw[thick] ([xshift=-2pt]v1.south) -- node[midway,left] {\scriptsize $\ell_1$} ([xshift=-2pt]v2.north);
    \draw[thick] ([xshift=2pt]v1.south) -- node[midway,right] {\scriptsize $\ell_2$} ([xshift=2pt]v2.north);

    \draw[->, red, thick] (v1) -- node[midway,above] {\scriptsize $\ell_3$} (v3);
    \draw[->, red, thick] (v4) -- node[midway,above] {\scriptsize $\ell_4$} (v2);
    \draw[thick] (v3) -- node[midway,right] {\scriptsize $\ell_5$} (v4);

    \draw[thick, black] ([yshift=2pt]v3.east) -- node[midway,above] {\scriptsize $\ell_6$} ([yshift=2pt]v5.west);
    \draw[thick] ([yshift=-2pt]v3.east) -- node[midway,below] {\scriptsize $\ell_7$} ([yshift=-2pt]v5.west);

    \draw[thick, red] ([yshift=2pt]v4.east) -- node[midway,above] {\scriptsize $\ell_8$} ([yshift=2pt]v6.west);
    \draw[thick] ([yshift=-2pt]v4.east) -- node[midway,below] {\scriptsize $\ell_9$} ([yshift=-2pt]v6.west);

    \draw[dashed] (2.2,2.8) rectangle (4.8,-0.8);
\end{tikzpicture}

\end{center}
\end{figure}
\textbf{Step 3D4G}: Two degree $3,4$ atoms $v_1,v_2$ are connected by a double bond. For details, see 9.3.6 in \cite{DH23}.

We remove from the molecule $v_1,v_2$, and we add to the acyclic graph $G$ the maximum number of bonds linked to $v_1,v_2$ such that there is no cycle created in the acyclic graph $G$.
\begin{figure}[H]
\centering
\begin{tikzpicture}[scale=1, every node/.style={draw, circle, thick, minimum size=5mm,inner sep=0pt}]
\node (v2) at (0,0) {\scriptsize $v_2$};
\node (v3) [shift={(-1,0)}] at (v2) {\scriptsize $v_3$};
\node (v1) [shift={(1,0)}] at (v2) {\scriptsize $v_1$};
\node (v4) [shift={(1.5,0.8)}] at (v2) {\scriptsize $v_4$};
\node (v5) [shift={(1.5,-0.8)}] at (v2) {\scriptsize $v_5$};
\draw[double,red] (v1) -- (v2);
\draw[-,red] (v2) -- (v3);
\draw[-,red] (v1) -- (v4);
\draw[-] (v1) -- (v5);
\end{tikzpicture}
\quad\quad
\begin{tikzpicture}[scale=1, every node/.style={draw, circle, thick, minimum size=5mm,inner sep=0pt}]
\node (v2) at (0,0) {\scriptsize $v_2$};
\node (v3) [shift={(-1,0)}] at (v2) {\scriptsize $v_3$};
\node (v1) [shift={(1,0)}] at (v2) {\scriptsize $v_1$};
\node (v4) [shift={(1.5,0.8)}] at (v2) {\scriptsize $v_4$};
\node (v5) [shift={(1.5,-0.8)}] at (v2) {\scriptsize $v_5$};
\draw[double,red] (v2) -- (v1);
\draw[-] (v2) -- (v3);
\draw[-,red] (v1) -- (v4);
\draw[-,red] (v1) -- (v5);
\end{tikzpicture}
\caption{On the left, we have added $(v_1,v_2)$, $(v_3,v_2)$, and $(v_1,v_4)$ to the acyclic graph $G$, and the other edges cannot be added due to a cycle. On the right, we have added $(v_1,v_2)$, $(v_1,v_4)$, and $(v_1,v_5)$, and we are in the case where the other edges produce a cycle.}
\end{figure}

\textbf{Step 3S2G}: Two degree $3,2$ atoms $v_1,v_2$ are connected by a single bond. For details see 9.3.7 in \cite{DH23}.

We remove from the molecule $v_1,v_2$, and we add to the acyclic graph $G$ the maximum number of bonds linked to $v_1,v_2$ such that there is no cycle created in the acyclic graph $G$.

\begin{figure}[H]
\centering
\begin{tikzpicture}[scale=1, every node/.style={draw, circle, thick, minimum size=5mm,inner sep=0pt}]
\node (v2) at (0,0) {\scriptsize $v_2$};
\node (v3) [shift={(-1,0)}] at (v2) {\scriptsize $v_3$};
\node (v1) [shift={(1,0)}] at (v2) {\scriptsize $v_1$};
\node (v4) [shift={(1.5,0.8)}] at (v2) {\scriptsize $v_4$};
\node (v5) [shift={(1.5,-0.8)}] at (v2) {\scriptsize $v_5$};
\draw[-,red] (v2) -- (v1);
\draw[-,red] (v2) -- (v3);
\draw[-,red] (v1) -- (v4);
\draw[-] (v1) -- (v5);
\end{tikzpicture}
\quad
\begin{tikzpicture}[scale=1, every node/.style={draw, circle, thick, minimum size=5mm,inner sep=0pt}]
\node (v2) at (0,0) {\scriptsize $v_2$};
\node (v3) [shift={(-1,0)}] at (v2) {\scriptsize $v_3$};
\node (v1) [shift={(1,0)}] at (v2) {\scriptsize $v_1$};
\node (v4) [shift={(1.5,0.8)}] at (v2) {\scriptsize $v_4$};
\node (v5) [shift={(1.5,-0.8)}] at (v2) {\scriptsize $v_5$};
\draw[red] (v2) -- (v1);
\draw[-] (v2) -- (v3);
\draw[-,red] (v1) -- (v4);
\draw[-,red] (v1) -- (v5);
\end{tikzpicture}
\caption{On the left, we have added $(v_1,v_2)$, $(v_3,v_2)$, and $(v_1,v_4)$ to the acyclic graph $G$, and the other edges cannot be added due to a cycle. On the right, we have added $(v_1,v_2)$, $(v_1,v_4),(v_1,v_5)$, and we are in the case where the other edges produce a cycle.}
\end{figure}

\textbf{Step 3R}: one atom $v$ of degree $3$ is connected with three atoms of degree $4$. For details see 9.3.8 in \cite{DH23}.
\begin{itemize}
\item Step (3R-1): We remove from the molecule the atom $v$ and its three bonds. We add to the acyclic graph $G$ the maximum number of bonds linked to $v$ and its neighbors such that there is no cycle created in $G$.
\begin{figure}[H]
\centering
\begin{tikzpicture}[scale=1, every node/.style={draw, circle, thick, minimum size=4mm,inner sep=0pt}]
\node (v) at (0,0) {\scriptsize $v$};
\node (v1) [shift={(-1,0)}] at (v) {\scriptsize $v_1$};
\node (v2) [shift={(1,0.5)}] at (v) {\scriptsize $v_2$};
\node (v3) [shift={(1,-0.5)}] at (v) {\scriptsize $v_3$};
\draw[-,red] (v) -- (v1);
\draw[-,red] (v) -- (v2);
\draw[-,red] (v) -- (v3);
\end{tikzpicture}
\caption{We provide an example. We have added $(v,v_1),(v,v_2),(v,v_3)$ to the acyclic graph $G$. The other bonds linked to $v_1,v_2,v_3$ cannot be added due to potential cycles not represented here.}
\end{figure}
\item Step (3R-2G): We remove from the molecule the atom $v$ and its three bonds. In addition, we remove from the molecule $v_1',v_2'$ and their bonds. We add to the acyclic graph $G$ the maximum number of bonds linked to $v,v_1',v_2'$ and its neighbors such that there is no cycle created in $G$.
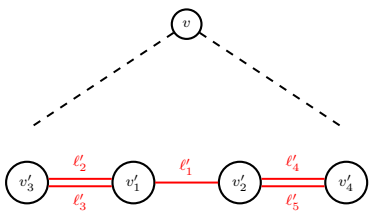
\begin{center}
\begin{tikzpicture}[scale=0.7, every node/.style={scale=0.7}]
    \node[draw, circle, thick] (v) at (0,3) {\scriptsize $v$};
    \node[draw, circle, thick] (v3) at (-3,0) {\scriptsize $v_3'$};
    \node[draw, circle, thick] (v1) at (-1,0) {\scriptsize $v_1'$};
    \node[draw, circle, thick] (v2) at (1,0) {\scriptsize $v_2'$};
    \node[draw, circle, thick] (v4) at (3,0) {\scriptsize $v_4'$};

    \draw[dashed] (v) -- (-3,1);
    \draw[dashed] (v) -- (3,1);

    \draw[red] ([yshift=2pt]v3.east) -- node[above]{\scriptsize $\ell_2'$} ([yshift=2pt]v1.west);
\draw[red] ([yshift=-2pt]v3.east) -- node[below]{\scriptsize $\ell_3'$} ([yshift=-2pt]v1.west);

\draw[red] (v1) -- node[above]{\scriptsize $\ell_1'$} (v2);

\draw[red] ([yshift=2pt]v2.east) -- node[above]{\scriptsize $\ell_4'$} ([yshift=2pt]v4.west);
\draw[red] ([yshift=-2pt]v2.east) -- node[below]{\scriptsize $\ell_5'$} ([yshift=-2pt]v4.west);
\end{tikzpicture}
\end{center}
\end{itemize}

\medskip

\textbf{Step 2R}: One atom $v$ of degree $2$ is connected to one or two atoms of degree $2$ or $4$.
\begin{itemize}
  \item Step (2R-1): we assume that $v$ is connected to a degree-$4$ atom $v'$ by a double bond, 
  where the two bonds have opposite directions. We remove from the molecule the atom $v$ and the double bond $(v,v')$. We add to the acyclic graph $G$ the maximum number of bonds linked to $v$ and its neighbors such that there is no cycle created in $G$.
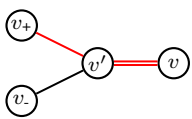
\begin{figure}[H]
\centering
\begin{tikzpicture}[scale=1, every node/.style={draw, circle, thick, minimum size=4mm,inner sep=0pt}]
\node (v) at (0,0) {\scriptsize $v$};
\node (v') [shift={(-1,0)}] at (v) {\scriptsize $v'$};
\node (v+) [shift={(-1,0.5)}] at (v') {\scriptsize $v_{\text{+}}$};
\node (v-) [shift={(-1,-0.5)}] at (v') {\scriptsize $v_{\text{-}}$};
\draw[double,red] (v) -- (v');
\draw[-,red] (v') -- (v+);
\draw[-] (v') -- (v-);
\end{tikzpicture}
\caption{We provide an example. We have added $(v,v')$ and $(v',v_{+})$ to the acyclic graph $G$. The other bond $(v',v_{+})$ cannot be added due to potential cycles not represented here.}
\end{figure}
  \item Step (2R-2G): We assume that $v$ is connected to a degree-$4$ atom by a double bond, 
where the two bonds have the same direction. We remove from the molecule the atom $v$ and the double bond. We add to the acyclic graph $G$ the maximum number of bonds linked to $v$ and its neighbors such that there is no cycle created in $G$.

\item Step (2R-3): We assume that $v$ is connected to a degree-$4$ atom $v'$ by a single bond $(v,v')$, 
and also connected to another degree-$2$ or degree-$4$ atom $v''$ by a single bond $(v,v'')$. 
We remove from the molecule the atom $v$ and the two bonds. We add to the acyclic graph $G$ the maximum number of bonds linked to $v$ and its neighbors such that there is no cycle created in $G$.
\begin{figure}[H]
\centering
\begin{tikzpicture}[scale=1, every node/.style={draw, circle, thick, minimum size=4mm,inner sep=0pt}]
\node (v) at (0,0) {\scriptsize $v$};
\node (v') [shift={(-1,0)}] at (v) {\scriptsize $v'$};
\node (v'') [shift={(1,0)}] at (v) {\scriptsize $v''$};
\node (v1) [shift={(0,1)}] at (v') {\scriptsize $v_1$};
\node (v2) [shift={(-1,0)}] at (v') {\scriptsize $v_2$};
\node (v3) [shift={(0,-1)}] at (v') {\scriptsize $v_3$};
\node (v4) [shift={(1,0.5)}] at (v'') {\scriptsize $v_4$};
\node (v5) [shift={(1,-0.5)}] at (v'') {\scriptsize $v_5$};
\draw[-,red] (v) -- (v');
\draw[-,red] (v) -- (v'');
\draw[-] (v') -- (v1);
\draw[-] (v') -- (v2);
\draw[-] (v') -- (v3);
\draw[-] (v'') -- (v4);
\draw[-] (v'') -- (v5);
\end{tikzpicture}
\caption{In the example above, we have added $(v,v')$ and $(v,v'')$ to the acyclic graph $G$. The other edges $(v',v_1),(v',v_2),(v',v_3),(v'',v_4),(v'',v_5)$ cannot be added due to potential cycles not represented here.}
\end{figure}
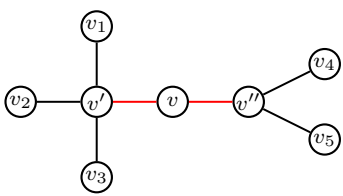
  \item Step (2R-4): We assume that $v$ is connected to two degree-$2$ atoms, $v'$ and $v''$, 
  by two single bonds, such that neither $v'$ nor $v''$ is connected to a degree-$3$ atom. 
  We remove from the molecule the atoms $v, v', v''$, as well as all bonds connected to them. We add to the acyclic graph $G$ the maximum number of bonds linked to $v,v',v''$ and their neighbors such that there is no cycle created in $G$.
\begin{figure}[H]
\centering
\begin{tikzpicture}[scale=1, every node/.style={draw, circle, thick, minimum size=4mm,inner sep=0pt}]
\node (v) at (0,0) {\scriptsize $v$};
\node (v') [shift={(-1,0)}] at (v) {\scriptsize $v'$};
\node (v'') [shift={(1,0)}] at (v) {\scriptsize $v''$};
\node (v1) [shift={(-1,0.5)}] at (v') {\scriptsize $v_1$};
\node (v2) [shift={(-1,-0.5)}] at (v') {\scriptsize $v_2$};
\node (v3) [shift={(1,0.5)}] at (v'') {\scriptsize $v_3$};
\node (v4) [shift={(1,-0.5)}] at (v'') {\scriptsize $v_4$};

\draw[-,red] (v) -- (v');
\draw[-,red] (v) -- (v'');
\draw[-] (v') -- (v1);
\draw[-] (v') -- (v2);
\draw[-] (v'') -- (v3);
\draw[-] (v'') -- (v4);
\end{tikzpicture}
\caption{We provide an example. We have added $(v,v')$ and $(v,v'')$ to the acyclic graph $G$. The other edges $(v',v_1),(v',v_2),(v'',v_4),(v'',v_5)$ cannot be added due to potential cycles not represented here.}
\end{figure}
  \item Step (2R-5): We assume that $v$ is connected to a degree-$2$ atom, $v'$, by a double bond. 
  We remove from the molecule the atoms $v, v'$ and the double bond. We add to the acyclic graph $G$ the maximum number of bonds linked to $v,v'$ and their neighbors such that there is no cycle created in $G$.
  \begin{figure}[H]
\centering
\begin{tikzpicture}[scale=1, every node/.style={draw, circle, thick, minimum size=4mm,inner sep=0pt}]
\node (v) at (0,0) {\scriptsize $v$};
\node (v') [shift={(-1,0)}] at (v) {\scriptsize $v'$};
\draw[double,red] (v) -- (v');
\end{tikzpicture}
\caption{We provide an example. We have added $(v,v')$ to the acyclic graph $G$. The other edges linked to $v,v'$ cannot be added due to potential cycles not represented here.}
\end{figure}
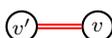
\end{itemize}
\end{proof}
\newpage
Here is a table which synthesises the evolution of quantities step by step. It gives the dynamic weights used in this Kruskal algorithm with the highest weight at the top:

\begin{center} \label{long_table}
\renewcommand{\arraystretch}{1.15}
\small
\begin{tabular}{|p{2.6cm}|c|p{2.7cm}|c|}
\hline
\textbf{Step} & $\Delta\chi$ & $\Delta\gamma$ & $\Delta\kappa$ \\
\hline
(DA)
& $\le -d(v)+j+h$
& $\begin{cases}
0,\\[0.2em]
-2+\dfrac{1}{4}
\end{cases}$
& $0$ \\
\hline
(3S3-1)
& $-2$
& $-2$
& $-1$ \\
\hline
(3S3-2G)
& $-2$
& $-2+\dfrac{1}{6(d-1)}$
& $-2$ \\
\hline
(3S3-3G)
& $-1$
& $-1+\dfrac{1}{6(d-1)}$
& $-1$ \\
\hline
(3S3-4G), (3S3-5G)
& $-3$ or $-2$
& ($-3$ or $-2$)$+\dfrac{1}{6(d-1)}$
& $-2$ \\
\hline
(3D3-1)
& $-2$
& $-2$
& $-1$ \\
\hline
(3D3-2G), (3D3-4G), (3D3-5G)
& $-2$
& $-2+\dfrac{1}{4(d-1)}$
& $-2$ \\
\hline
(3D3-3G)
& $-1$
& $-1+\dfrac{1}{4(d-1)}$
& $-1$ \\
\hline
(3D3-6G)
& $-5$ or $-4$
& ($-5$ or $-4$)$+\dfrac{1}{6(d-1)}$
& $-4$ \\
\hline
(3D4G)
& $-3$
& $-3+\dfrac{1}{4(d-1)}$
& $-2$ \\
\hline
(3S2G)
& $-2$
& $-2+\dfrac{1}{4(d-1)}$
& $-2$ \\
\hline
(3R-1)
& $-2$
& $-2$
& $-1$ \\
\hline
(3R-2G)
& $-5$ or $-4$
& ($-5$ or $-4$)$+\dfrac{1}{4(d-1)}$
& $-4$ \\
\hline
(2R-1), (2R-3)
& $-1$
& $-1$
& $-1$ \\
\hline
(2R-2G)
& $-1$
& $-1+\dfrac{1}{3(d-1)}$
& $-1$ \\
\hline
(2R-4)
& $-1$ or $-1$
& $-1$
& $-1$ \\
\hline
(2R-5)
& $-1$
& $-1$
& $-1$ \\
\hline
\end{tabular}
\end{center}
\subsection{Algorithm and proof of Theorem \ref{main_theorem}}
\label{Sec::3:2}

We start by describing the algorithm given in \cite{DH23} and we use all the steps introduced in the previous section.
The algorithm is described in two phases. First, we remove degenerate atoms using steps (DA). In the second phase, there is therefore no more degenerate atom, and it should be noted that none of the steps can create a potentially degenerate atom.

In this second phase, the algorithm is described as a large loop. Once inside it, certain rules must be followed in order, depending on the molecule $\mathbf{M}$, to choose the next step (i) or define a control point to choose between two possibilities (ii) for the next step. It is also possible in some cases to choose more than one step or control point successively, respecting specific rules, until the end of the loop execution and returning to its beginning. The loop ends when $\mathbf{M}$ contains only isolated atoms.

\subsubsection*{Phase 1: elimination of degenerate atoms}

Using (DA) repeatedly, eliminate all degenerate atoms.

We note that these steps do not create new degenerate atoms of degree $4$. At the end of phase 1, there will be no more degenerate atoms. This property will be preserved for the rest of the algorithm.

\subsubsection*{Phase 2: description of the loop}

Note that there is no triple bond at the beginning.

\begin{enumerate}
  \item If $\mathbf{M}$ contains a bridge, it must be removed using (BR). Repeat until $\mathbf{M}$ contains no more bridges.

  \item If $\mathbf{M}$ contains two degree-$3$ atoms, $v_1$ and $v_2$, connected by a single bond $l_1$, then:
  \begin{enumerate}
    \item If $\mathbf{M}$ contains a functional group, execute (3S3-5G). Return to (1).
    \item If $\mathbf{M}$ contains the functional group in the figure and satisfies (i) and (ii), with $d(v_3) = \dots = d(v_6) = 4$, a control point is needed: choose between (3S3-1) and (3S3-2G). Return to (1).
    \item If (i) and (ii) are satisfied, but $d(v_3)$ and $d(v_5)$ are not equal to $4$, a control point is needed: choose between (3S3-2G) and (3S3-3G). If a triple bond forms between $v_3$ and $v_5$ after (3S3-3G), remove it using (TB-1)–(TB-2). Return to (1).
    \item If one of conditions (i) or (ii) is not satisfied, execute (3S3-4G). Return to (1).
  \end{enumerate}

  \item If $\mathbf{M}$ contains two degree-$3$ atoms, $v_1$ and $v_2$, connected by a double bond $(l_1, l_2)$, then:
  \begin{enumerate}
    \item If $\mathbf{M}$ contains the functional group corresponding to (3D3-4G) or (3D3-5G), perform the corresponding step. Return to (1).
    \item If $\mathbf{M}$ contains the functional group corresponding to (3D3-1)–(3D3-3G), we are at the beginning of a type II chain. If this chain continues, with $v_3$ and $v_4$ connected by a double bond and also connected to $v_5, v_6$ by opposite single bonds, a control point is needed: choose between (3D3-1) and (3D3-2G). Return to (1).
    \item If the type II chain does not continue:
    \begin{enumerate}[label=(\roman*)]
      \item If all atoms other than $(v_1,v_2)$ do not have degree $4$, a control point is needed: choose between (3D3-2G) and (3D3-3G). If a triple bond forms between $v_3$ and $v_4$ after (3D3-3G), remove it using (TB-1)–(TB-2). Return to (1).
      \item If $v_3$ and $v_4$ are as in Figure 30, execute (3D3-6G). Return to (1).
      \item Otherwise, a control point is needed: choose between (3D3-1) and (3D3-2G). Return to (1).
    \end{enumerate}
  \end{enumerate}

  \item If $\mathbf{M}$ contains a degree-$3$ atom $v_1$ connected to a degree-$4$ atom $v_2$ by a double bond $(\ell_1,\ell_2)$, perform (3D4G). Return to (1).

  \item If $\mathbf{M}$ contains a degree-$3$ atom $v_1$ connected to a degree-$2$ atom $v_2$, perform (3S2G). Return to (1).

  \item If $\mathbf{M}$ contains a degree-$3$ atom $v$ connected to three degree-$4$ atoms $v_j$ ($j=1,2,3$) by three single bonds $\ell_j$, then:
  \begin{enumerate}
    \item If the component without $v$ and $\ell_j$ contains a special bond, execute (3R-2G). Return to (1).
    \item Otherwise, execute (3R-1). Return to (1).
  \end{enumerate}

  \item $\mathbf{M}$ must contain only atoms of degree $0$, $2$, and $4$. If we are in one of the cases (2R-2G) to (2R-5), perform the corresponding step. Return to (1).

  \item If $\mathbf{M}$ contains a degree-$2$ atom $v$ connected to a degree-$4$ atom $v_1$ by a double bond with opposite directions, we are at the beginning of a type I chain. Without requiring its continuation (unlike (3-b)), execute (2R-1) until the end of the chain. Return to (7), remembering to analyse this component.
\end{enumerate}

\begin{proof}[of Theorem  \ref{main_theorem}] At the beginning of the algorithm the acyclic graph $G$ associated with the molecule $\mathbf{M}$ has no edges but contains all the atoms of the molecule $\mathbf{M}$. Then, we let the algorithm run, which ends when all the atoms are removed. Proposition \ref{acyclic_steps} guarantees that at each step the graph $G$ is acyclic. In the end, we have included in $G$ all edges that do not create a cycle. Therefore, one obtains a spanning tree: if $G$ is not connected, then we must have missed some bond during the steps of the algorithm which does not create a cycle in $G$, which is impossible. Note that this is a Kruskal-style algorithm as we select the edges and the nodes according to the order given above. An important point is that we may have multiple components during the building of the spanning tree, contrary to Prim's algorithm, for instance. We provide an example in the next section that shows multiple components in the construction of the spanning tree.
\end{proof}
\newpage
\subsection{Example of spanning tree}

Here is an example of molecule reduction, along with the construction of the associated spanning tree. The construction of the spanning tree heavily depends on the choice of the order of the bridge removals, especially regarding the selection of additional edges.
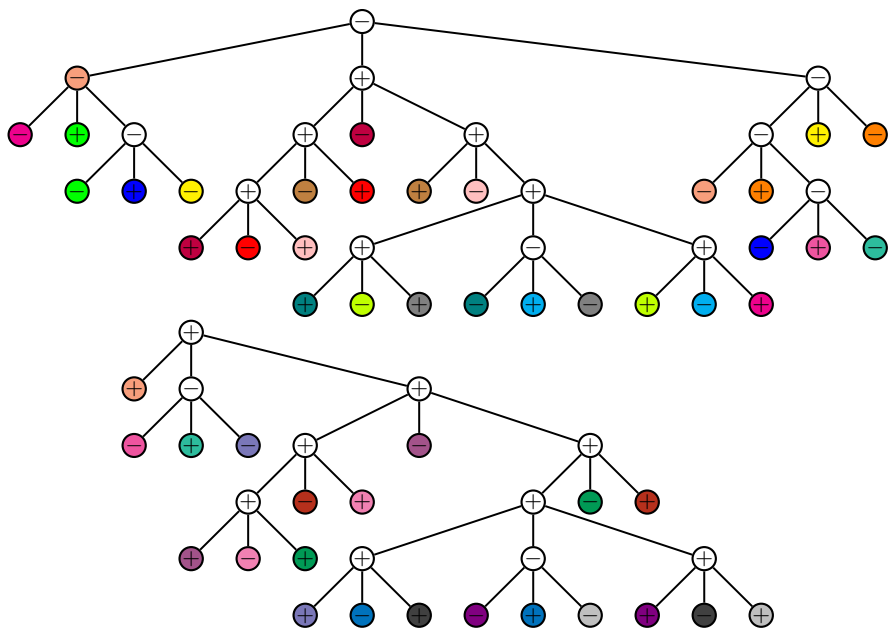
\begin{figure}[H]
\centering
\begin{tikzpicture}[scale=0.6, every node/.style={draw, circle, thick, minimum size=2mm,inner sep=0pt}]
    \node (-1) at (-4,0) {\scriptsize $-$};
\node (-11) [shift={(-3.75,-0.75)},fill=Melon] at (-1) {\scriptsize $-$};
\node (-12) [shift={(0,-0.75)}] at (-1) {\scriptsize $+$};
\node (-13) [shift={(6, -0.75)}] at (-1) {\scriptsize $-$};
\draw[-] (-1) -- (-11);
\draw[-] (-1) -- (-12);
\draw[-] (-1) -- (-13);

\node (-111) [shift={(-0.75,-0.75)},fill=magenta] at (-11) {\scriptsize $-$};
\node (-112) [shift={(0,-0.75)},fill=green] at (-11) {\scriptsize $+$};
\node (-113) [shift={(0.75,-0.75)}] at (-11) {\scriptsize $-$};
\draw[-] (-11) -- (-111);
\draw[-] (-11) -- (-112);
\draw[-] (-11) -- (-113);

\node (-121) [shift={(-0.75,-0.75)}] at (-12) {\scriptsize $+$};
\node (-122) [shift={(0,-0.75)},fill=purple] at (-12) {\scriptsize $-$};
\node (-123) [shift={(1.5,-0.75)}] at (-12) {\scriptsize $+$};
\draw[-] (-12) -- (-121);
\draw[-] (-12) -- (-122);
\draw[-] (-12) -- (-123);

\node (-131) [shift={(-0.75,-0.75)}] at (-13) {\scriptsize $-$};
\node (-132) [shift={(0,-0.75)},fill=yellow] at (-13) {\scriptsize $+$};
\node (-133) [shift={(0.75,-0.75)},fill=orange] at (-13) {\scriptsize $-$};
\draw[-] (-13) -- (-131);
\draw[-] (-13) -- (-132);
\draw[-] (-13) -- (-133);

\node (-1131) [shift={(-0.75,-0.75)},fill=green] at (-113) {\scriptsize $-$};
\node (-1132) [shift={(0,-0.75)},fill=blue] at (-113) {\scriptsize $+$};
\node (-1133) [shift={(0.75,-0.75)},fill=yellow] at (-113) {\scriptsize $-$};
\draw[-] (-113) -- (-1131);
\draw[-] (-113) -- (-1132);
\draw[-] (-113) -- (-1133);

\node (-1211) [shift={(-0.75,-0.75)}] at (-121) {\scriptsize $+$};
\node (-1212) [shift={(0,-0.75)},fill=brown] at (-121) {\scriptsize $-$};
\node (-1213) [shift={(0.75,-0.75)},fill=red] at (-121) {\scriptsize $+$};
\draw[-] (-121) -- (-1211);
\draw[-] (-121) -- (-1212);
\draw[-] (-121) -- (-1213);

\node (-1231) [shift={(-0.75,-0.75)},fill=brown] at (-123) {\scriptsize $+$};
\node (-1232) [shift={(0,-0.75)},fill=pink] at (-123) {\scriptsize $-$};
\node (-1233) [shift={(0.75,-0.75)}] at (-123) {\scriptsize $+$};
\draw[-] (-123) -- (-1231);
\draw[-] (-123) -- (-1232);
\draw[-] (-123) -- (-1233);

\node (-1311) [shift={(-0.75,-0.75)},fill=Melon] at (-131) {\scriptsize $-$};
\node (-1312) [shift={(0,-0.75)},fill=orange] at (-131) {\scriptsize $+$};
\node (-1313) [shift={(0.75,-0.75)}] at (-131) {\scriptsize $-$};
\draw[-] (-131) -- (-1311);
\draw[-] (-131) -- (-1312);
\draw[-] (-131) -- (-1313);

\node (-13131) [shift={(-0.75,-0.75)},fill=blue] at (-1313) {\scriptsize $-$};
\node (-13132) [shift={(0,-0.75)},fill=Rhodamine] at (-1313) {\scriptsize $+$};
\node (-13133) [shift={(0.75,-0.75)},fill=SeaGreen] at (-1313) {\scriptsize $-$};
\draw[-] (-1313) -- (-13131);
\draw[-] (-1313) -- (-13132);
\draw[-] (-1313) -- (-13133);

\node (-12111) [shift={(-0.75,-0.75)},fill=purple] at (-1211) {\scriptsize $+$};
\node (-12112) [shift={(0,-0.75)},fill=red] at (-1211) {\scriptsize $-$};
\node (-12113) [shift={(0.75,-0.75)},fill=pink] at (-1211) {\scriptsize $+$};
\draw[-] (-1211) -- (-12111);
\draw[-] (-1211) -- (-12112);
\draw[-] (-1211) -- (-12113);

\node (-12331) [shift={(-2.25,-0.75)}] at (-1233) {\scriptsize $+$};
\node (-12332) [shift={(0,-0.75)}] at (-1233) {\scriptsize $-$};
\node (-12333) [shift={(2.25,-0.75)}] at (-1233) {\scriptsize $+$};
\draw[-] (-1233) -- (-12331);
\draw[-] (-1233) -- (-12332);
\draw[-] (-1233) -- (-12333);

\node (-123311) [shift={(-0.75,-0.75)},fill=teal] at (-12331) {\scriptsize $+$};
\node (-123312) [shift={(0,-0.75)},fill=lime] at (-12331) {\scriptsize $-$};
\node (-123313) [shift={(0.75,-0.75)},fill=gray] at (-12331) {\scriptsize $+$};
\draw[-] (-12331) -- (-123311);
\draw[-] (-12331) -- (-123312);
\draw[-] (-12331) -- (-123313);

\node (-123321) [shift={(-0.75,-0.75)},fill=teal] at (-12332) {\scriptsize $-$};
\node (-123322) [shift={(0,-0.75)},fill=cyan] at (-12332) {\scriptsize $+$};
\node (-123323) [shift={(0.75,-0.75)},fill=gray] at (-12332) {\scriptsize $-$};
\draw[-] (-12332) -- (-123321);
\draw[-] (-12332) -- (-123322);
\draw[-] (-12332) -- (-123323);

\node (-123331) [shift={(-0.75,-0.75)},fill=lime] at (-12333) {\scriptsize $+$};
\node (-123332) [shift={(0,-0.75)},fill=cyan] at (-12333) {\scriptsize $-$};
\node (-123333) [shift={(0.75,-0.75)},fill=magenta] at (-12333) {\scriptsize $+$};
\draw[-] (-12333) -- (-123331);
\draw[-] (-12333) -- (-123332);
\draw[-] (-12333) -- (-123333);
\end{tikzpicture}

\begin{tikzpicture}[scale=0.6, every node/.style={draw, circle, thick, minimum size=2mm,inner sep=0pt}]
    \node (1) at (-4,-6) {\scriptsize $+$};
\node (11) [shift={(-0.75,-0.75)},fill=Melon] at (1) {\scriptsize $+$};
\node (12) [shift={(0,-0.75)}] at (1) {\scriptsize $-$};
\node (13) [shift={(3,-0.75)}] at (1) {\scriptsize $+$};
\draw[-] (1) -- (11);
\draw[-] (1) -- (12);
\draw[-] (1) -- (13);

\node (121) [shift={(-0.75,-0.75)},fill=Rhodamine] at (12) {\scriptsize $-$};
\node (122) [shift={(0,-0.75)},fill=SeaGreen] at (12) {\scriptsize $+$};
\node (123) [shift={(0.75,-0.75)},fill=Periwinkle] at (12) {\scriptsize $-$};
\draw[-] (121) -- (12);
\draw[-] (122) -- (12);
\draw[-] (123) -- (12);

\node (131) [shift={(-1.5,-0.75)}] at (13) {\scriptsize $+$};
\node (132) [shift={(0,-0.75)},fill=DarkOrchid] at (13) {\scriptsize $-$};
\node (133) [shift={(2.25,-0.75)}] at (13) {\scriptsize $+$};
\draw[-] (131) -- (13);
\draw[-] (132) -- (13);
\draw[-] (133) -- (13);

\node (1311) [shift={(-0.75,-0.75)}] at (131) {\scriptsize $+$};
\node (1312) [shift={(0,-0.75)},fill=BrickRed] at (131) {\scriptsize $-$};
\node (1313) [shift={(0.75,-0.75)},fill=CarnationPink] at (131) {\scriptsize $+$};
\draw[-] (1311) -- (131);
\draw[-] (1312) -- (131);
\draw[-] (1313) -- (131);

\node (1331) [shift={(-0.75,-0.75)}] at (133) {\scriptsize $+$};
\node (1332) [shift={(0,-0.75)},fill=ForestGreen] at (133) {\scriptsize $-$};
\node (1333) [shift={(0.75,-0.75)},fill=BrickRed] at (133) {\scriptsize $+$};
\draw[-] (1331) -- (133);
\draw[-] (1332) -- (133);
\draw[-] (1333) -- (133);

\node (13111) [shift={(-0.75,-0.75)},fill=DarkOrchid] at (1311) {\scriptsize $+$};
\node (13112) [shift={(0,-0.75)},fill=CarnationPink] at (1311) {\scriptsize $-$};
\node (13113) [shift={(0.75,-0.75)},fill=ForestGreen] at (1311) {\scriptsize $+$};
\draw[-] (13111) -- (1311);
\draw[-] (13112) -- (1311);
\draw[-] (13113) -- (1311);

\node (13311) [shift={(-2.25,-0.75)}] at (1331) {\scriptsize $+$};
\node (13312) [shift={(0,-0.75)}] at (1331) {\scriptsize $-$};
\node (13313) [shift={(2.25,-0.75)}] at (1331) {\scriptsize $+$};
\draw[-] (13311) -- (1331);
\draw[-] (13312) -- (1331);
\draw[-] (13313) -- (1331);

\node (133111) [shift={(-0.75,-0.75)},fill=Periwinkle] at (13311) {\scriptsize $+$};
\node (133112) [shift={(0,-0.75)},fill=RoyalBlue] at (13311) {\scriptsize $-$};
\node (133113) [shift={(0.75,-0.75)},fill=darkgray] at (13311) {\scriptsize $+$};
\draw[-] (133111) -- (13311);
\draw[-] (133112) -- (13311);
\draw[-] (133113) -- (13311);

\node (133121) [shift={(-0.75,-0.75)},fill=violet] at (13312) {\scriptsize $-$};
\node (133122) [shift={(0,-0.75)},fill=RoyalBlue] at (13312) {\scriptsize $+$};
\node (133123) [shift={(0.75,-0.75)},fill=lightgray] at (13312) {\scriptsize $-$};
\draw[-] (133121) -- (13312);
\draw[-] (133122) -- (13312);
\draw[-] (133123) -- (13312);

\node (133131) [shift={(-0.75,-0.75)},fill=violet] at (13313) {\scriptsize $+$};
\node (133132) [shift={(0,-0.75)},fill=darkgray] at (13313) {\scriptsize $-$};
\node (133133) [shift={(0.75,-0.75)},fill=lightgray] at (13313) {\scriptsize $+$};
\draw[-] (133131) -- (13313);
\draw[-] (133132) -- (13313);
\draw[-] (133133) -- (13313);
\end{tikzpicture}
\caption{Initial couple}
\end{figure}
\begin{figure}[H]
\centering
\begin{tikzpicture}[scale=0.6, every node/.style={draw, circle, thick, minimum size=5mm,inner sep=0pt}]
    \node (-1t) at (-2,1) {\scriptsize -$1t$};
\node (-2t) at (-4,1) {\scriptsize -$2t$};
\node (-3t) at (-6,2) {\scriptsize -$3t$};
\node (-4t) at (-6,0) {\scriptsize -$4t$};
\node (-1b) at (-2,-3) {\scriptsize -$1b$};
\node (-2b) at (-4,-3) {\scriptsize -$2b$};
\node (-3b) at (-6,-4) {\scriptsize -$3b$};
\node (-4b) at (-6,-2) {\scriptsize -$4b$};
\node (1t) at (0,0) {\scriptsize $1t$};
\node (1b) at (0,-2) {\scriptsize $1b$};
\node (2t) at (2,0) {\scriptsize $2t$};
\node (2b) at (2,-2) {\scriptsize $2b$};
\node (3t) at (4,0) {\scriptsize $3t$};
\node (3b) at (4,-2) {\scriptsize $3b$};
\node (4t) at (6,0) {\scriptsize $4t$};
\node (4b) at (6,-2) {\scriptsize $4b$};
\node (+1t) at (8,1) {\scriptsize +$1t$};
\node (+2t) at (10,1) {\scriptsize +$2t$};
\node (+3t) at (12,2) {\scriptsize +$3t$};
\node (+4t) at (12,0) {\scriptsize +$4t$};
\node (+1b) at (8,-3) {\scriptsize +$1b$};
\node (+2b) at (10,-3) {\scriptsize +$2b$};
\node (+4b) at (12,-2) {\scriptsize +$4b$};
\node (+3b) at (12,-4) {\scriptsize +$3b$};

    \draw[->] (-1t) -- (1t);
    \draw[->] (1t) -- (1b);
    \draw[<->,green] (2b) -- (1b);
    \draw[->] (1t) -- (2t);
    \draw[->,yellow] (2b) -- (2t);
    \draw[->, black, bend left=5] (2t) to (3t);
    \draw[->, orange, bend right=5] (2t) to (3t);
    \draw[->,blue] (3b) -- (2b);
    \draw[->,purple] (-1t) -- (-2t);
    \draw[->, red, bend left=5] (-2t) to (-3t);
    \draw[->, black, bend right=5] (-2t) to (-3t);
    \draw[->] (-4t) -- (-1t);
    \draw[->,brown] (-3t) -- (-4t);
    \draw[->] (-3t) -- (-1t);
    \draw[->,pink] (-4t) -- (-2t);
    \draw[->] (-4b) -- (-4t);
    \draw[<-,lime] (-1b) -- (-2b);
    
    \draw[->, teal, bend left=5] (-2b) to (-3b);
    \draw[<-, gray, bend right=5] (-2b) to (-3b);
    \draw[<-] (-4b) -- (-1b);
    \draw[<-] (-3b) -- (-4b);
    \draw[<-,cyan] (-3b) -- (-1b);
    \draw[<-] (-4b) -- (-2b);
    \draw[->,magenta] (1b) -- (-1b);
    \draw[->] (3t) -- (3b);
    \draw[->] (4t) -- (4b);
    \draw[<-,Melon] (4t) -- (3t);
    \draw[->, SeaGreen, bend left=5] (3b) to (4b);
    \draw[<-, Rhodamine, bend right=5] (3b) to (4b);
    \draw[->] (+1t) -- (+2t);
    \draw[->, CarnationPink, bend left=5] (+2t) to (+3t);
    \draw[->, black, bend right=5] (+2t) to (+3t);
    \draw[->] (+4t) -- (+1t);
    \draw[->,BrickRed] (+3t) -- (+4t);
    \draw[->] (+3t) -- (+1t);
    \draw[->,ForestGreen] (+4t) -- (+2t);
    \draw[->] (+4b) -- (+4t);
    \draw[<-,darkgray] (+1b) -- (+2b);
    \draw[<-, violet, bend left=5] (+2b) to (+3b);
    \draw[<-, lightgray, bend right=5] (+2b) to (+3b);
    \draw[<-] (+4b) -- (+1b);
    \draw[<-] (+3b) -- (+4b);
    \draw[<-,RoyalBlue] (+3b) -- (+1b);
    \draw[<-] (+4b) -- (+2b);
    \draw[->] (+1t) -- (4t);
    \draw[->,Periwinkle] (4b) -- (+1b);
\end{tikzpicture}
\caption{Molecule associated to the initial couple}
\end{figure}
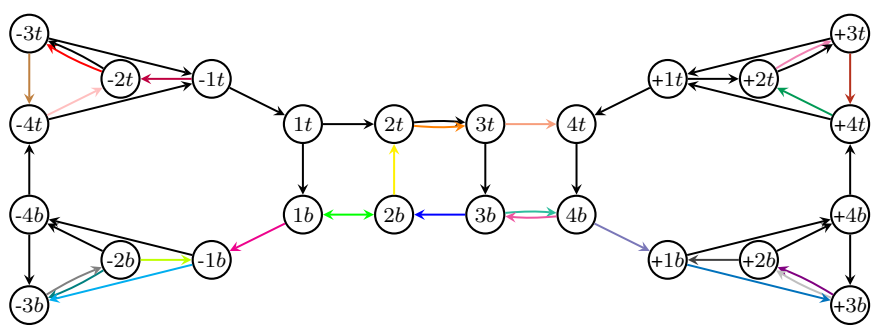

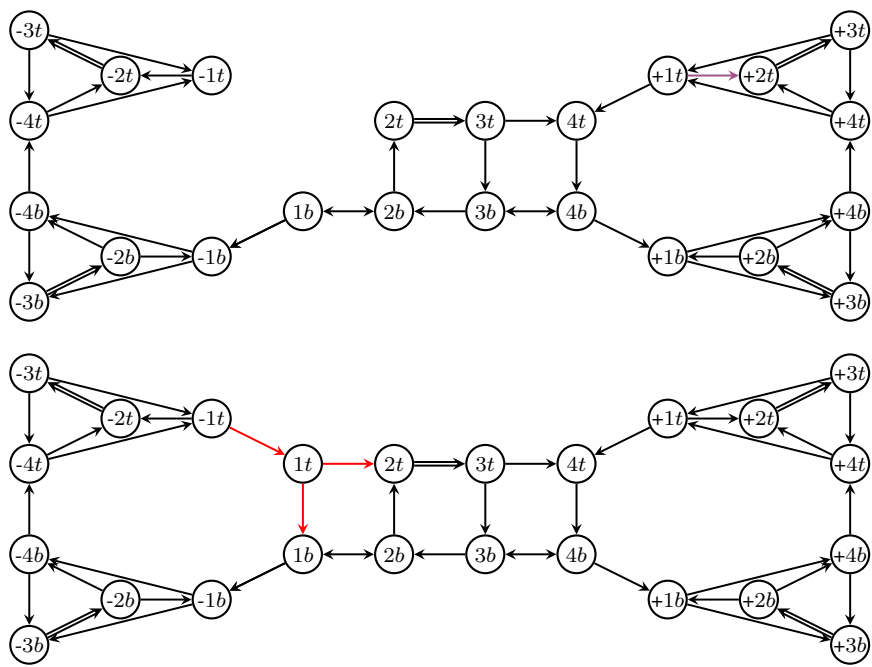
\begin{figure}[H]
\centering
\begin{tikzpicture}[scale=0.6, every node/.style={draw, circle, radius =1cm, thick, minimum size=5mm,inner sep=0pt}]
\node (-1t) at (-2,1) {\scriptsize{-$1t$}};
\node (-2t) at (-4,1) {\scriptsize{-$2t$}};
\node (-3t) at (-6,2) {\scriptsize{-$3t$}};
\node (-4t) at (-6,0) {\scriptsize{-$4t$}};
\node (-1b) at (-2,-3) {\scriptsize{-$1b$}};
\node (-2b) at (-4,-3) {\scriptsize{-$2b$}};
\node (-3b) at (-6,-4) {\scriptsize{-$3b$}};
\node (-4b) at (-6,-2) {\scriptsize{-$4b$}};
\node (1b) at (0,-2) {\scriptsize{$1b$}};
\node (2t) at (2,0) {\scriptsize{$2t$}};
\node (2b) at (2,-2) {\scriptsize{$2b$}};
\node (3t) at (4,0) {\scriptsize{$3t$}};
\node (3b) at (4,-2) {\scriptsize{$3b$}};
\node (4t) at (6,0) {\scriptsize{$4t$}};
\node (4b) at (6,-2) {\scriptsize{$4b$}};
\node (+1t) at (8,1) {\scriptsize{+$1t$}};
\node (+2t) at (10,1) {\scriptsize{+$2t$}};
\node (+3t) at (12,2) { \scriptsize{+$3t$}};
\node (+4t) at (12,0) {\scriptsize{+$4t$}};
\node (+1b) at (8,-3) {\scriptsize{+$1b$}};
\node (+2b) at (10,-3) {\scriptsize{+$2b$}};
\node (+4b) at (12,-2) {\scriptsize{+$4b$}};
\node (+3b) at (12,-4) {\scriptsize{+$3b$}};
	
	\draw[<->] (2b) -- (1b);
	\draw[->] (2b) -- (2t);
	\draw[->,double] (2t) -- (3t);
	\draw[->] (3b) -- (2b);
	\draw[->] (-1t) -- (-2t);
	\draw[->,double] (-2t) -- (-3t);
	\draw[->] (-4t) -- (-1t);
	\draw[->] (-3t) -- (-4t);
	\draw[->] (-3t) -- (-1t);
	\draw[->] (-4t) -- (-2t);
	\draw[->] (-4b) -- (-4t);
	\draw[<-] (-1b) -- (-2b);
	\draw[<-,double] (-2b) -- (-3b);
	\draw[<-] (-4b) -- (-1b);
	\draw[<-] (-3b) -- (-4b);
	\draw[<-] (-3b) -- (-1b);
	\draw[<-] (-4b) -- (-2b);
	\draw[->] (1b) -- (-1b);
	\draw[->] (1b) -- (-1b);
	\draw[->] (3t) -- (3b);
	\draw[->] (4t) -- (4b);
	\draw[<-] (4t) -- (3t);
	\draw[<->] (3b) -- (4b);
	\draw[->,DarkOrchid] (+1t) -- (+2t);
	\draw[->,double] (+2t) -- (+3t);
	\draw[->] (+4t) -- (+1t);
	\draw[->] (+3t) -- (+4t);
	\draw[->] (+3t) -- (+1t);
	\draw[->] (+4t) -- (+2t);
	\draw[->] (+4b) -- (+4t);
	\draw[<-] (+1b) -- (+2b);
	\draw[<-,double] (+2b) -- (+3b);
	\draw[<-] (+4b) -- (+1b);
	\draw[<-] (+3b) -- (+4b);
	\draw[<-] (+3b) -- (+1b);
	\draw[<-] (+4b) -- (+2b);
	\draw[->] (+1t) -- (4t);
	\draw[->] (4b) -- (+1b);
\end{tikzpicture}
\\
\vspace{1.0em}
\begin{tikzpicture}[scale=0.6, every node/.style={draw, circle, thick, minimum size=5mm,inner sep=0pt}]
	\node (-1t) at (-2,1) {\scriptsize{-$1t$}};
	\node (-2t) at (-4,1) {\scriptsize{-$2t$}};
	\node (-3t) at (-6,2) {\scriptsize{-$3t$}};
	\node (-4t) at (-6,0) {\scriptsize{-$4t$}};
	\node (1t) at (0,0) {\scriptsize{$1t$}};
	\node (-1b) at (-2,-3) {\scriptsize{-$1b$}};
	\node (-2b) at (-4,-3) {\scriptsize{-$2b$}};
	\node (-3b) at (-6,-4) {\scriptsize{-$3b$}};
	\node (-4b) at (-6,-2) {\scriptsize{-$4b$}};
	\node (1b) at (0,-2) {\scriptsize{$1b$}};
	\node (2t) at (2,0) {\scriptsize{$2t$}};
	\node (2b) at (2,-2) {\scriptsize{$2b$}};
	\node (3t) at (4,0) {\scriptsize{$3t$}};
	\node (3b) at (4,-2) {\scriptsize{$3b$}};
	\node (4t) at (6,0) {\scriptsize{$4t$}};
	\node (4b) at (6,-2) {\scriptsize{$4b$}};
	\node (+1t) at (8,1) {\scriptsize{+$1t$}};
	\node (+2t) at (10,1) {\scriptsize{+$2t$}};
	\node (+3t) at (12,2) { \scriptsize{+$3t$}};
	\node (+4t) at (12,0) {\scriptsize{+$4t$}};
	\node (+1b) at (8,-3) {\scriptsize{+$1b$}};
	\node (+2b) at (10,-3) {\scriptsize{+$2b$}};
	\node (+4b) at (12,-2) {\scriptsize{+$4b$}};
	\node (+3b) at (12,-4) {\scriptsize{+$3b$}};
	
	\draw[->,red] (-1t) -- (1t);
	\draw[->,red] (1t) -- (1b);
	\draw[<->] (2b) -- (1b);
	\draw[->,red] (1t) -- (2t);
	\draw[->] (2b) -- (2t);
	\draw[->,double] (2t) -- (3t);
	\draw[->] (3b) -- (2b);
	\draw[->] (-1t) -- (-2t);
	\draw[->,double] (-2t) -- (-3t);
	\draw[->] (-4t) -- (-1t);
	\draw[->] (-3t) -- (-4t);
	\draw[->] (-3t) -- (-1t);
	\draw[->] (-4t) -- (-2t);
	\draw[->] (-4b) -- (-4t);
	\draw[<-] (-1b) -- (-2b);
	\draw[<-,double] (-2b) -- (-3b);
	\draw[<-] (-4b) -- (-1b);
	\draw[<-] (-3b) -- (-4b);
	\draw[<-] (-3b) -- (-1b);
	\draw[<-] (-4b) -- (-2b);
	\draw[->] (1b) -- (-1b);
	\draw[->] (1b) -- (-1b);
	\draw[->] (3t) -- (3b);
	\draw[->] (4t) -- (4b);
	\draw[<-] (4t) -- (3t);
	\draw[<->] (3b) -- (4b);
	\draw[->] (+1t) -- (+2t);
	\draw[->,double] (+2t) -- (+3t);
	\draw[->] (+4t) -- (+1t);
	\draw[->] (+3t) -- (+4t);
	\draw[->] (+3t) -- (+1t);
	\draw[->] (+4t) -- (+2t);
	\draw[->] (+4b) -- (+4t);
	\draw[<-] (+1b) -- (+2b);
	\draw[<-,double] (+2b) -- (+3b);
	\draw[<-] (+4b) -- (+1b);
	\draw[<-] (+3b) -- (+4b);
	\draw[<-] (+3b) -- (+1b);
	\draw[<-] (+4b) -- (+2b);
	\draw[->] (+1t) -- (4t);
	\draw[->] (4b) -- (+1b);
\end{tikzpicture}
\caption{Molecule and spanning tree after step 1. We apply $3R1$ to the atom $(1t)$.}
\end{figure}

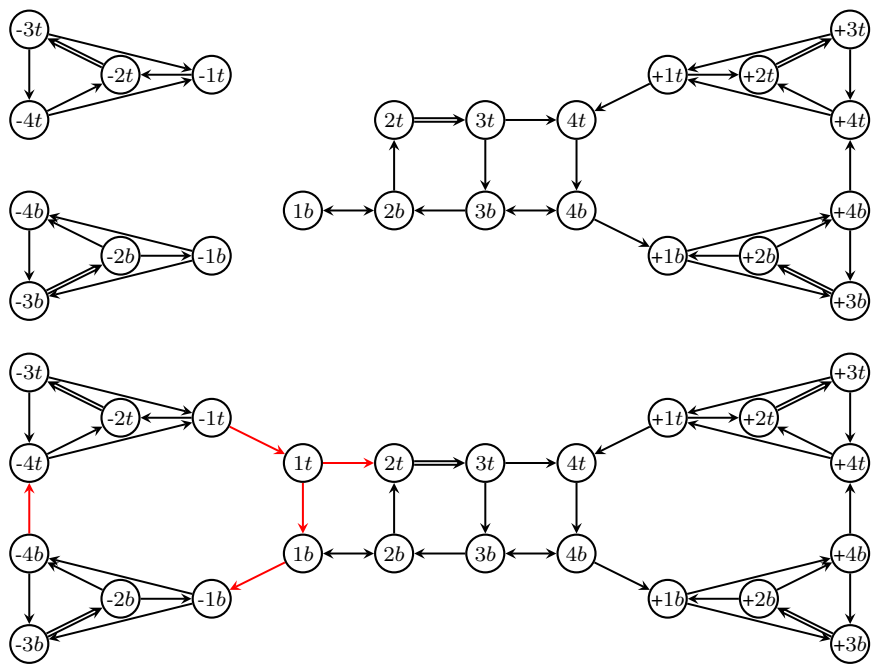
\begin{figure}[H]
	\centering
	\begin{tikzpicture}[scale=0.6, every node/.style={draw, circle, thick, minimum size=5mm,inner sep=0pt}]
	\node (-1t) at (-2,1) {\scriptsize{-$1t$}};
	\node (-2t) at (-4,1) {\scriptsize{-$2t$}};
	\node (-3t) at (-6,2) {\scriptsize{-$3t$}};
	\node (-4t) at (-6,0) {\scriptsize{-$4t$}};
	\node (-1b) at (-2,-3) {\scriptsize{-$1b$}};
	\node (-2b) at (-4,-3) {\scriptsize{-$2b$}};
	\node (-3b) at (-6,-4) {\scriptsize{-$3b$}};
	\node (-4b) at (-6,-2) {\scriptsize{-$4b$}};
	\node (1b) at (0,-2) {\scriptsize{$1b$}};
	\node (2t) at (2,0) {\scriptsize{$2t$}};
	\node (2b) at (2,-2) {\scriptsize{$2b$}};
	\node (3t) at (4,0) {\scriptsize{$3t$}};
	\node (3b) at (4,-2) {\scriptsize{$3b$}};
	\node (4t) at (6,0) {\scriptsize{$4t$}};
	\node (4b) at (6,-2) {\scriptsize{$4b$}};
	\node (+1t) at (8,1) {\scriptsize{+$1t$}};
	\node (+2t) at (10,1) {\scriptsize{+$2t$}};
	\node (+3t) at (12,2) { \scriptsize{+$3t$}};
	\node (+4t) at (12,0) {\scriptsize{+$4t$}};
	\node (+1b) at (8,-3) {\scriptsize{+$1b$}};
	\node (+2b) at (10,-3) {\scriptsize{+$2b$}};
	\node (+4b) at (12,-2) {\scriptsize{+$4b$}};
	\node (+3b) at (12,-4) {\scriptsize{+$3b$}};
		
		\draw[<->] (2b) -- (1b);
		\draw[->] (2b) -- (2t);
		\draw[->,double] (2t) -- (3t);
		\draw[->] (3b) -- (2b);
		\draw[->] (-1t) -- (-2t);
		\draw[->,double] (-2t) -- (-3t);
		\draw[->] (-4t) -- (-1t);
		\draw[->] (-3t) -- (-4t);
		\draw[->] (-3t) -- (-1t);
		\draw[->] (-4t) -- (-2t);
		\draw[<-] (-1b) -- (-2b);
		\draw[<-,double] (-2b) -- (-3b);
		\draw[<-] (-4b) -- (-1b);
		\draw[<-] (-3b) -- (-4b);
		\draw[<-] (-3b) -- (-1b);
		\draw[<-] (-4b) -- (-2b);
		\draw[->] (3t) -- (3b);
		\draw[->] (4t) -- (4b);
		\draw[<-] (4t) -- (3t);
		\draw[<->] (3b) -- (4b);
		\draw[->] (+1t) -- (+2t);
		\draw[->,double] (+2t) -- (+3t);
		\draw[->] (+4t) -- (+1t);
		\draw[->] (+3t) -- (+4t);
		\draw[->] (+3t) -- (+1t);
		\draw[->] (+4t) -- (+2t);
		\draw[->] (+4b) -- (+4t);
		\draw[<-] (+1b) -- (+2b);
		\draw[<-,double] (+2b) -- (+3b);
		\draw[<-] (+4b) -- (+1b);
		\draw[<-] (+3b) -- (+4b);
		\draw[<-] (+3b) -- (+1b);
		\draw[<-] (+4b) -- (+2b);
		\draw[->] (+1t) -- (4t);
		\draw[->] (4b) -- (+1b);
	\end{tikzpicture}
\\
\vspace{1.0em}
	\begin{tikzpicture}[scale=0.6, every node/.style={draw, circle, thick, minimum size=5mm,inner sep=0pt}]
		\node (-1t) at (-2,1) {\scriptsize{-$1t$}};
	\node (-2t) at (-4,1) {\scriptsize{-$2t$}};
	\node (-3t) at (-6,2) {\scriptsize{-$3t$}};
	\node (-4t) at (-6,0) {\scriptsize{-$4t$}};
	\node (1t) at (0,0) {\scriptsize{$1t$}};
	\node (-1b) at (-2,-3) {\scriptsize{-$1b$}};
	\node (-2b) at (-4,-3) {\scriptsize{-$2b$}};
	\node (-3b) at (-6,-4) {\scriptsize{-$3b$}};
	\node (-4b) at (-6,-2) {\scriptsize{-$4b$}};
	\node (1b) at (0,-2) {\scriptsize{$1b$}};
	\node (2t) at (2,0) {\scriptsize{$2t$}};
	\node (2b) at (2,-2) {\scriptsize{$2b$}};
	\node (3t) at (4,0) {\scriptsize{$3t$}};
	\node (3b) at (4,-2) {\scriptsize{$3b$}};
	\node (4t) at (6,0) {\scriptsize{$4t$}};
	\node (4b) at (6,-2) {\scriptsize{$4b$}};
	\node (+1t) at (8,1) {\scriptsize{+$1t$}};
	\node (+2t) at (10,1) {\scriptsize{+$2t$}};
	\node (+3t) at (12,2) { \scriptsize{+$3t$}};
	\node (+4t) at (12,0) {\scriptsize{+$4t$}};
	\node (+1b) at (8,-3) {\scriptsize{+$1b$}};
	\node (+2b) at (10,-3) {\scriptsize{+$2b$}};
	\node (+4b) at (12,-2) {\scriptsize{+$4b$}};
	\node (+3b) at (12,-4) {\scriptsize{+$3b$}};
		
		\draw[->,red] (-1t) -- (1t);
		\draw[->,red] (1t) -- (1b);
		\draw[<->] (2b) -- (1b);
		\draw[->,red] (1t) -- (2t);
		\draw[->] (2b) -- (2t);
		\draw[->,double] (2t) -- (3t);
		\draw[->] (3b) -- (2b);
		\draw[->] (-1t) -- (-2t);
		\draw[->,double] (-2t) -- (-3t);
		\draw[->] (-4t) -- (-1t);
		\draw[->] (-3t) -- (-4t);
		\draw[->] (-3t) -- (-1t);
		\draw[->] (-4t) -- (-2t);
		\draw[->,red] (-4b) -- (-4t);
		\draw[<-] (-1b) -- (-2b);
		\draw[<-,double] (-2b) -- (-3b);
		\draw[<-] (-4b) -- (-1b);
		\draw[<-] (-3b) -- (-4b);
		\draw[<-] (-3b) -- (-1b);
		\draw[<-] (-4b) -- (-2b);
		\draw[->,red] (1b) -- (-1b);
		\draw[->] (3t) -- (3b);
		\draw[->] (4t) -- (4b);
		\draw[<-] (4t) -- (3t);
		\draw[<->] (3b) -- (4b);
		\draw[->] (+1t) -- (+2t);
		\draw[->,double] (+2t) -- (+3t);
		\draw[->] (+4t) -- (+1t);
		\draw[->] (+3t) -- (+4t);
		\draw[->] (+3t) -- (+1t);
		\draw[->] (+4t) -- (+2t);
		\draw[->] (+4b) -- (+4t);
		\draw[<-] (+1b) -- (+2b);
		\draw[<-,double] (+2b) -- (+3b);
		\draw[<-] (+4b) -- (+1b);
		\draw[<-] (+3b) -- (+4b);
		\draw[<-] (+3b) -- (+1b);
		\draw[<-] (+4b) -- (+2b);
		\draw[->] (+1t) -- (4t);
		\draw[->] (4b) -- (+1b);
	\end{tikzpicture}
		\caption{Molecule and spanning tree after step 2 and step 3. We remove the bridges $(1b,\text{-}1b)$ and $(\text{-}4b,\text{-}4t)$.}
\end{figure}

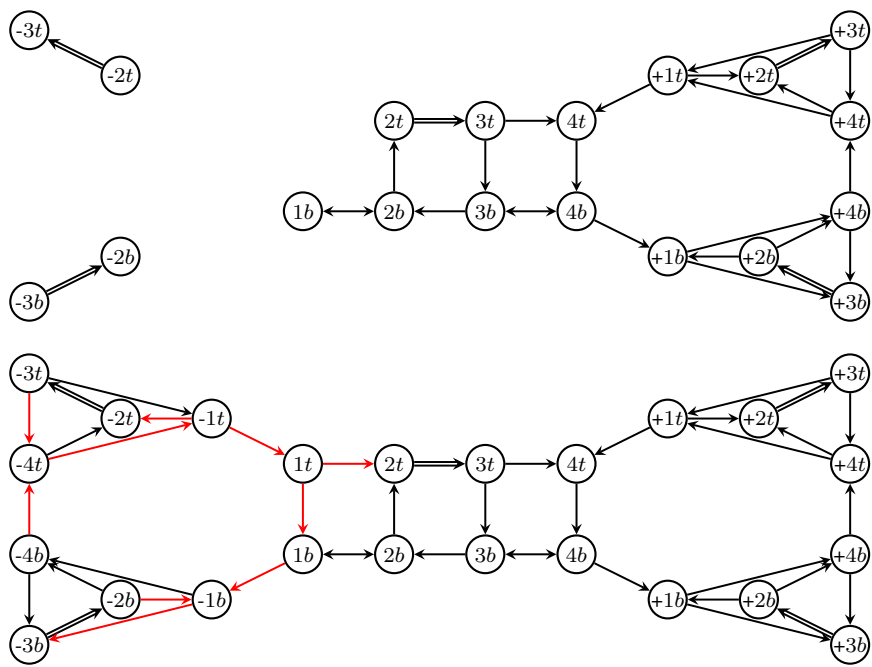
\begin{figure}[H]
	\centering
	\begin{tikzpicture}[scale=0.6, every node/.style={draw, circle, thick, minimum size=5mm,inner sep=0pt}]
		\node (-2t) at (-4,1) {\scriptsize -$2t$};
\node (-3t) at (-6,2) {\scriptsize -$3t$};
\node (-2b) at (-4,-3) {\scriptsize -$2b$};
\node (-3b) at (-6,-4) {\scriptsize -$3b$};
\node (1b) at (0,-2) {\scriptsize $1b$};
\node (2t) at (2,0) {\scriptsize $2t$};
\node (2b) at (2,-2) {\scriptsize $2b$};
\node (3t) at (4,0) {\scriptsize $3t$};
\node (3b) at (4,-2) {\scriptsize $3b$};
\node (4t) at (6,0) {\scriptsize $4t$};
\node (4b) at (6,-2) {\scriptsize $4b$};
\node (+1t) at (8,1) {\scriptsize +$1t$};
\node (+2t) at (10,1) {\scriptsize +$2t$};
\node (+3t) at (12,2) {\scriptsize +$3t$};
\node (+4t) at (12,0) {\scriptsize +$4t$};
\node (+1b) at (8,-3) {\scriptsize +$1b$};
\node (+2b) at (10,-3) {\scriptsize +$2b$};
\node (+4b) at (12,-2) {\scriptsize +$4b$};
\node (+3b) at (12,-4) {\scriptsize +$3b$};
		
		\draw[<->] (2b) -- (1b);
		\draw[->] (2b) -- (2t);
		\draw[->,double] (2t) -- (3t);
		\draw[->] (3b) -- (2b);
		\draw[->,double] (-2t) -- (-3t);
		\draw[<-,double] (-2b) -- (-3b);
		\draw[->] (3t) -- (3b);
		\draw[->] (4t) -- (4b);
		\draw[<-] (4t) -- (3t);
		\draw[<->] (3b) -- (4b);
		\draw[->] (+1t) -- (+2t);
		\draw[->,double] (+2t) -- (+3t);
		\draw[->] (+4t) -- (+1t);
		\draw[->] (+3t) -- (+4t);
		\draw[->] (+3t) -- (+1t);
		\draw[->] (+4t) -- (+2t);
		\draw[->] (+4b) -- (+4t);
		\draw[<-] (+1b) -- (+2b);
		\draw[<-,double] (+2b) -- (+3b);
		\draw[<-] (+4b) -- (+1b);
		\draw[<-] (+3b) -- (+4b);
		\draw[<-] (+3b) -- (+1b);
		\draw[<-] (+4b) -- (+2b);
		\draw[->] (+1t) -- (4t);
		\draw[->] (4b) -- (+1b);
	\end{tikzpicture}
	\\
	\vspace{1.0em}
	\begin{tikzpicture}[scale=0.6, every node/.style={draw, circle, thick, minimum size=5mm,inner sep=0pt}]
		\node (-1t) at (-2,1) {\scriptsize{-$1t$}};
		\node (-2t) at (-4,1) {\scriptsize{-$2t$}};
		\node (-3t) at (-6,2) {\scriptsize{-$3t$}};
		\node (-4t) at (-6,0) {\scriptsize{-$4t$}};
		\node (-1b) at (-2,-3) {\scriptsize{-$1b$}};
		\node (-2b) at (-4,-3) {\scriptsize{-$2b$}};
		\node (-3b) at (-6,-4) {\scriptsize{-$3b$}};
		\node (-4b) at (-6,-2) {\scriptsize{-$4b$}};
		\node (1t) at (0,0) {\scriptsize{$1t$}};
		\node (1b) at (0,-2) {\scriptsize{$1b$}};
		\node (2t) at (2,0) {\scriptsize{$2t$}};
		\node (2b) at (2,-2) {\scriptsize{$2b$}};
		\node (3t) at (4,0) {\scriptsize{$3t$}};
		\node (3b) at (4,-2) {\scriptsize{$3b$}};
		\node (4t) at (6,0) {\scriptsize{$4t$}};
		\node (4b) at (6,-2) {\scriptsize{$4b$}};
		\node (+1t) at (8,1) {\scriptsize{+$1t$}};
		\node (+2t) at (10,1) {\scriptsize{+$2t$}};
		\node (+3t) at (12,2) {\scriptsize{+$3t$}};
		\node (+4t) at (12,0) {\scriptsize{+$4t$}};
		\node (+1b) at (8,-3) {\scriptsize{+$1b$}};
		\node (+2b) at (10,-3) {\scriptsize{+$2b$}};
		\node (+4b) at (12,-2) {\scriptsize{+$4b$}};
		\node (+3b) at (12,-4) {\scriptsize{+$3b$}};
		
		\draw[->,red] (-1t) -- (1t);
		\draw[->,red] (1t) -- (1b);
		\draw[<->] (2b) -- (1b);
		\draw[->,red] (1t) -- (2t);
		\draw[->] (2b) -- (2t);
		\draw[->,double] (2t) -- (3t);
		\draw[->] (3b) -- (2b);
		\draw[->,red] (-1t) -- (-2t);
		\draw[->,double] (-2t) -- (-3t);
		\draw[->,red] (-4t) -- (-1t);
		\draw[->,red] (-3t) -- (-4t);
		\draw[->] (-3t) -- (-1t);
		\draw[->] (-4t) -- (-2t);
		\draw[->,red] (-4b) -- (-4t);
		\draw[<-,red] (-1b) -- (-2b);
		\draw[<-,double] (-2b) -- (-3b);
		\draw[<-] (-4b) -- (-1b);
		\draw[<-] (-3b) -- (-4b);
		\draw[<-,red] (-3b) -- (-1b);
		\draw[<-] (-4b) -- (-2b);
		\draw[->,red] (1b) -- (-1b);
		\draw[->] (3t) -- (3b);
		\draw[->] (4t) -- (4b);
		\draw[<-] (4t) -- (3t);
		\draw[<->] (3b) -- (4b);
		\draw[->] (+1t) -- (+2t);
		\draw[->,double] (+2t) -- (+3t);
		\draw[->] (+4t) -- (+1t);
		\draw[->] (+3t) -- (+4t);
		\draw[->] (+3t) -- (+1t);
		\draw[->] (+4t) -- (+2t);
		\draw[->] (+4b) -- (+4t);
		\draw[<-] (+1b) -- (+2b);
		\draw[<-,double] (+2b) -- (+3b);
		\draw[<-] (+4b) -- (+1b);
		\draw[<-] (+3b) -- (+4b);
		\draw[<-] (+3b) -- (+1b);
		\draw[<-] (+4b) -- (+2b);
		\draw[->] (+1t) -- (4t);
		\draw[->] (4b) -- (+1b);
	\end{tikzpicture}
	\caption{Molecule and spanning tree after steps 4 and 5. We apply twice (3S3-5G), on $(\text{-}1t,\text{-}4t)$ and $(\text{-}1b,\text{-}4b)$.}
\end{figure}

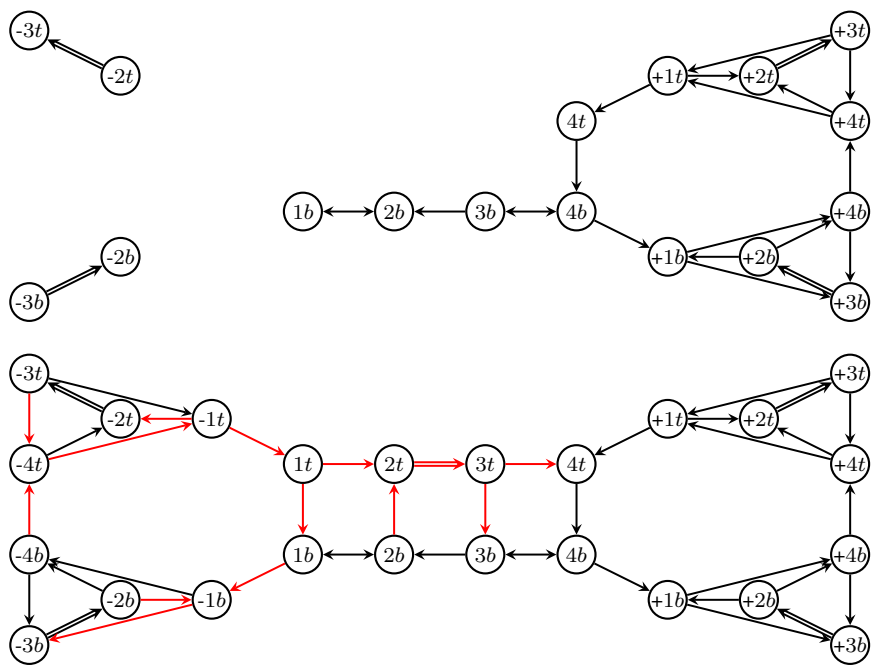
\begin{figure}[H]
	\centering
	\begin{tikzpicture}[scale=0.6, every node/.style={draw, circle, thick, minimum size=5mm,inner sep=0pt}]
		\node (-2t) at (-4,1) {\scriptsize -$2t$};
\node (-3t) at (-6,2) {\scriptsize -$3t$};
\node (-2b) at (-4,-3) {\scriptsize -$2b$};
\node (-3b) at (-6,-4) {\scriptsize -$3b$};
\node (1b) at (0,-2) {\scriptsize $1b$};
\node (2b) at (2,-2) {\scriptsize $2b$};
\node (3b) at (4,-2) {\scriptsize $3b$};
\node (4t) at (6,0) {\scriptsize $4t$};
\node (4b) at (6,-2) {\scriptsize $4b$};
\node (+1t) at (8,1) {\scriptsize +$1t$};
\node (+2t) at (10,1) {\scriptsize +$2t$};
\node (+3t) at (12,2) {\scriptsize +$3t$};
\node (+4t) at (12,0) {\scriptsize +$4t$};
\node (+1b) at (8,-3) {\scriptsize +$1b$};
\node (+2b) at (10,-3) {\scriptsize +$2b$};
\node (+4b) at (12,-2) {\scriptsize +$4b$};
\node (+3b) at (12,-4) {\scriptsize +$3b$};
		
		\draw[<->] (2b) -- (1b);
		\draw[->] (3b) -- (2b);
		\draw[->,double] (-2t) -- (-3t);
		\draw[<-,double] (-2b) -- (-3b);
		\draw[->] (4t) -- (4b);
		\draw[<->] (3b) -- (4b);
		\draw[->] (+1t) -- (+2t);
		\draw[->,double] (+2t) -- (+3t);
		\draw[->] (+4t) -- (+1t);
		\draw[->] (+3t) -- (+4t);
		\draw[->] (+3t) -- (+1t);
		\draw[->] (+4t) -- (+2t);
		\draw[->] (+4b) -- (+4t);
		\draw[<-] (+1b) -- (+2b);
		\draw[<-,double] (+2b) -- (+3b);
		\draw[<-] (+4b) -- (+1b);
		\draw[<-] (+3b) -- (+4b);
		\draw[<-] (+3b) -- (+1b);
		\draw[<-] (+4b) -- (+2b);
		\draw[->] (+1t) -- (4t);
		\draw[->] (4b) -- (+1b);
	\end{tikzpicture}
	\\
	\vspace{1.0em}
	\begin{tikzpicture}[scale=0.6, every node/.style={draw, circle, thick, minimum size=5mm,inner sep=0pt}]
		\node (-1t) at (-2,1) {\scriptsize -$1t$};
\node (-2t) at (-4,1) {\scriptsize -$2t$};
\node (-3t) at (-6,2) {\scriptsize -$3t$};
\node (-4t) at (-6,0) {\scriptsize -$4t$};
\node (-1b) at (-2,-3) {\scriptsize -$1b$};
\node (-2b) at (-4,-3) {\scriptsize -$2b$};
\node (-3b) at (-6,-4) {\scriptsize -$3b$};
\node (-4b) at (-6,-2) {\scriptsize -$4b$};
\node (1t) at (0,0) {\scriptsize $1t$};
\node (1b) at (0,-2) {\scriptsize $1b$};
\node (2t) at (2,0) {\scriptsize $2t$};
\node (2b) at (2,-2) {\scriptsize $2b$};
\node (3t) at (4,0) {\scriptsize $3t$};
\node (3b) at (4,-2) {\scriptsize $3b$};
\node (4t) at (6,0) {\scriptsize $4t$};
\node (4b) at (6,-2) {\scriptsize $4b$};
\node (+1t) at (8,1) {\scriptsize +$1t$};
\node (+2t) at (10,1) {\scriptsize +$2t$};
\node (+3t) at (12,2) {\scriptsize +$3t$};
\node (+4t) at (12,0) {\scriptsize +$4t$};
\node (+1b) at (8,-3) {\scriptsize +$1b$};
\node (+2b) at (10,-3) {\scriptsize +$2b$};
\node (+4b) at (12,-2) {\scriptsize +$4b$};
\node (+3b) at (12,-4) {\scriptsize +$3b$};
		
		\draw[->,red] (-1t) -- (1t);
		\draw[->,red] (1t) -- (1b);
		\draw[<->] (2b) -- (1b);
		\draw[->,red] (1t) -- (2t);
		\draw[->,red] (2b) -- (2t);
		\draw[->,double,red] (2t) -- (3t);
		\draw[->] (3b) -- (2b);
		\draw[->,red] (-1t) -- (-2t);
		\draw[->,double] (-2t) -- (-3t);
		\draw[->,red] (-4t) -- (-1t);
		\draw[->,red] (-3t) -- (-4t);
		\draw[->] (-3t) -- (-1t);
		\draw[->] (-4t) -- (-2t);
		\draw[->,red] (-4b) -- (-4t);
		\draw[<-,red] (-1b) -- (-2b);
		\draw[<-,double] (-2b) -- (-3b);
		\draw[<-] (-4b) -- (-1b);
		\draw[<-] (-3b) -- (-4b);
		\draw[<-,red] (-3b) -- (-1b);
		\draw[<-] (-4b) -- (-2b);
		\draw[->,red] (1b) -- (-1b);
		\draw[->,red] (3t) -- (3b);
		\draw[->] (4t) -- (4b);
		\draw[<-,red] (4t) -- (3t);
		\draw[<->] (3b) -- (4b);
		\draw[->] (+1t) -- (+2t);
		\draw[->,double] (+2t) -- (+3t);
		\draw[->] (+4t) -- (+1t);
		\draw[->] (+3t) -- (+4t);
		\draw[->] (+3t) -- (+1t);
		\draw[->] (+4t) -- (+2t);
		\draw[->] (+4b) -- (+4t);
		\draw[<-] (+1b) -- (+2b);
		\draw[<-,double] (+2b) -- (+3b);
		\draw[<-] (+4b) -- (+1b);
		\draw[<-] (+3b) -- (+4b);
		\draw[<-] (+3b) -- (+1b);
		\draw[<-] (+4b) -- (+2b);
		\draw[->] (+1t) -- (4t);
		\draw[->] (4b) -- (+1b);
	\end{tikzpicture}
	\caption{Molecule and spanning tree after step 6. We apply (3D4G) on $(2t,3t)$.}
\end{figure}

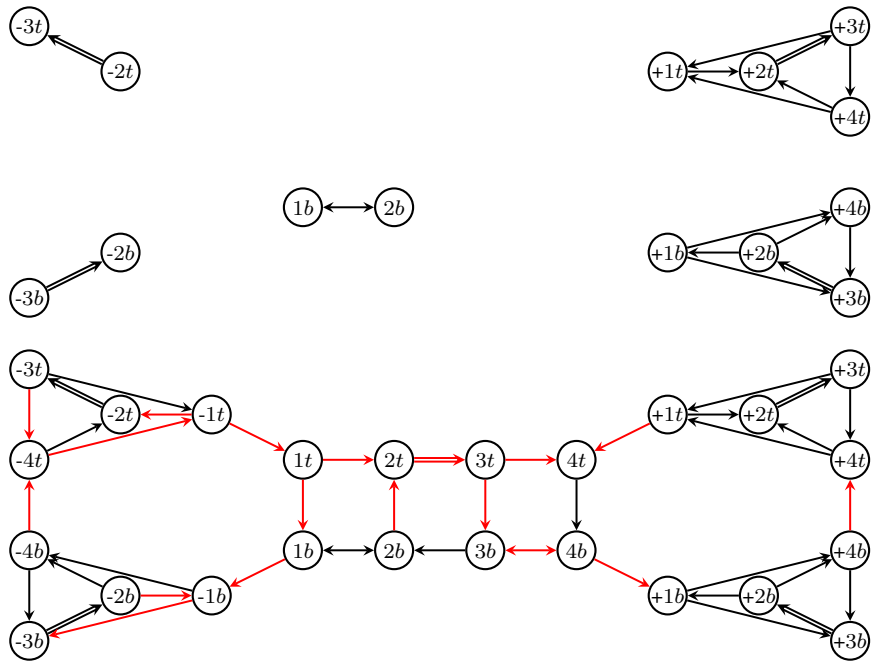
\begin{figure}[H]
	\centering
	\begin{tikzpicture}[scale=0.6, every node/.style={draw, circle, thick, minimum size=5mm,inner sep=0pt}]
		\node (-2t) at (-4,1) {\scriptsize -$2t$};
\node (-3t) at (-6,2) {\scriptsize -$3t$};
\node (-2b) at (-4,-3) {\scriptsize -$2b$};
\node (-3b) at (-6,-4) {\scriptsize -$3b$};
\node (1b) at (0,-2) {\scriptsize $1b$};
\node (2b) at (2,-2) {\scriptsize $2b$};
\node (+1t) at (8,1) {\scriptsize +$1t$};
\node (+2t) at (10,1) {\scriptsize +$2t$};
\node (+3t) at (12,2) {\scriptsize +$3t$};
\node (+4t) at (12,0) {\scriptsize +$4t$};
\node (+1b) at (8,-3) {\scriptsize +$1b$};
\node (+2b) at (10,-3) {\scriptsize +$2b$};
\node (+4b) at (12,-2) {\scriptsize +$4b$};
\node (+3b) at (12,-4) {\scriptsize +$3b$};
		
		\draw[<->] (2b) -- (1b);
		\draw[->,double] (-2t) -- (-3t);
		\draw[<-,double] (-2b) -- (-3b);
		\draw[->] (+1t) -- (+2t);
		\draw[->,double] (+2t) -- (+3t);
		\draw[->] (+4t) -- (+1t);
		\draw[->] (+3t) -- (+4t);
		\draw[->] (+3t) -- (+1t);
		\draw[->] (+4t) -- (+2t);
		\draw[<-] (+1b) -- (+2b);
		\draw[<-,double] (+2b) -- (+3b);
		\draw[<-] (+4b) -- (+1b);
		\draw[<-] (+3b) -- (+4b);
		\draw[<-] (+3b) -- (+1b);
		\draw[<-] (+4b) -- (+2b);
	\end{tikzpicture}
	\\
	\vspace{1.0em}
	\begin{tikzpicture}[scale=0.6, every node/.style={draw, circle, thick, minimum size=5mm,inner sep=0pt}]
		\node (-1t) at (-2,1) {\scriptsize -$1t$};
\node (-2t) at (-4,1) {\scriptsize -$2t$};
\node (-3t) at (-6,2) {\scriptsize -$3t$};
\node (-4t) at (-6,0) {\scriptsize -$4t$};
\node (-1b) at (-2,-3) {\scriptsize -$1b$};
\node (-2b) at (-4,-3) {\scriptsize -$2b$};
\node (-3b) at (-6,-4) {\scriptsize -$3b$};
\node (-4b) at (-6,-2) {\scriptsize -$4b$};
\node (1t) at (0,0) {\scriptsize $1t$};
\node (1b) at (0,-2) {\scriptsize $1b$};
\node (2t) at (2,0) {\scriptsize $2t$};
\node (2b) at (2,-2) {\scriptsize $2b$};
\node (3t) at (4,0) {\scriptsize $3t$};
\node (3b) at (4,-2) {\scriptsize $3b$};
\node (4t) at (6,0) {\scriptsize $4t$};
\node (4b) at (6,-2) {\scriptsize $4b$};
\node (+1t) at (8,1) {\scriptsize +$1t$};
\node (+2t) at (10,1) {\scriptsize +$2t$};
\node (+3t) at (12,2) {\scriptsize +$3t$};
\node (+4t) at (12,0) {\scriptsize +$4t$};
\node (+1b) at (8,-3) {\scriptsize +$1b$};
\node (+2b) at (10,-3) {\scriptsize +$2b$};
\node (+4b) at (12,-2) {\scriptsize +$4b$};
\node (+3b) at (12,-4) {\scriptsize +$3b$};
		
		\draw[->,red] (-1t) -- (1t);
		\draw[->,red] (1t) -- (1b);
		\draw[<->] (2b) -- (1b);
		\draw[->,red] (1t) -- (2t);
		\draw[->,red] (2b) -- (2t);
		\draw[->,double,red] (2t) -- (3t);
		\draw[->] (3b) -- (2b);
		\draw[->,red] (-1t) -- (-2t);
		\draw[->,double] (-2t) -- (-3t);
		\draw[->,red] (-4t) -- (-1t);
		\draw[->,red] (-3t) -- (-4t);
		\draw[->] (-3t) -- (-1t);
		\draw[->] (-4t) -- (-2t);
		\draw[->,red] (-4b) -- (-4t);
		\draw[<-,red] (-1b) -- (-2b);
		\draw[<-,double] (-2b) -- (-3b);
		\draw[<-] (-4b) -- (-1b);
		\draw[<-] (-3b) -- (-4b);
		\draw[<-,red] (-3b) -- (-1b);
		\draw[<-] (-4b) -- (-2b);
		\draw[->,red] (1b) -- (-1b);
		\draw[->,red] (3t) -- (3b);
		\draw[->] (4t) -- (4b);
		\draw[<-,red] (4t) -- (3t);
		\draw[<->,red] (3b) -- (4b);
		\draw[->] (+1t) -- (+2t);
		\draw[->,double] (+2t) -- (+3t);
		\draw[->] (+4t) -- (+1t);
		\draw[->] (+3t) -- (+4t);
		\draw[->] (+3t) -- (+1t);
		\draw[->] (+4t) -- (+2t);
		\draw[->,red] (+4b) -- (+4t);
		\draw[<-] (+1b) -- (+2b);
		\draw[<-,double] (+2b) -- (+3b);
		\draw[<-] (+4b) -- (+1b);
		\draw[<-] (+3b) -- (+4b);
		\draw[<-] (+3b) -- (+1b);
		\draw[<-] (+4b) -- (+2b);
		\draw[->,red] (+1t) -- (4t);
		\draw[->,red] (4b) -- (+1b);
	\end{tikzpicture}
	\caption{Molecule and spanning tree after steps 7-12. We apply BR on $(3b,2b)$, 2R-1 on $ 3b $, BR on $(\text{+}1t,4t)$, BR on $ (4t,4b) $, BR on $ (4b,\text{+}1b) $ and BR on $(\text{+}4b,\text{+}4t)$.}
\end{figure}

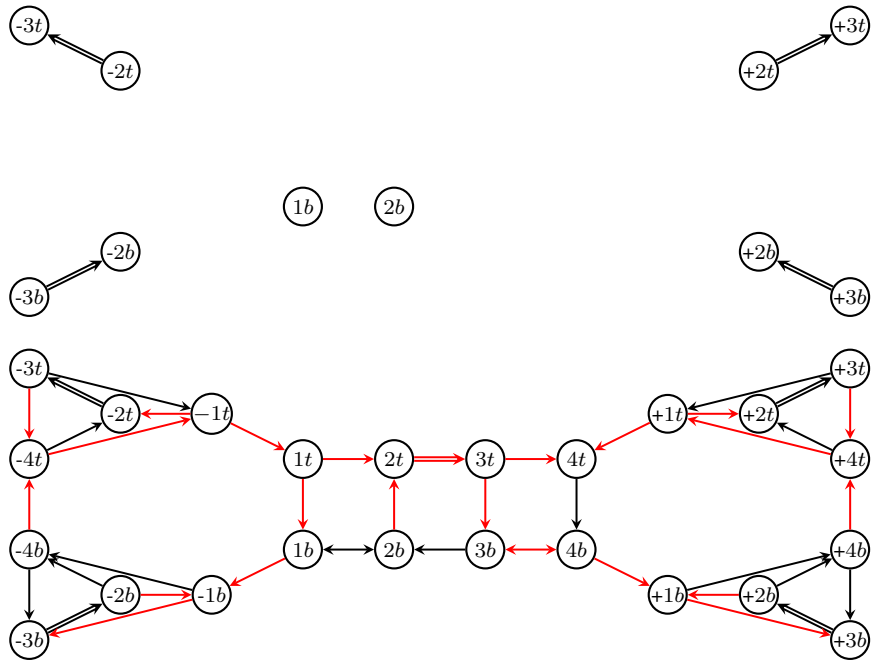
\begin{figure}[H]
	\centering
	\begin{tikzpicture}[scale=0.6, every node/.style={draw, circle, thick, minimum size=5mm,inner sep=0pt}]
		\node (-2t) at (-4,1) {\scriptsize -$2t$};
\node (-3t) at (-6,2) {\scriptsize -$3t$};
\node (-2b) at (-4,-3) {\scriptsize -$2b$};
\node (-3b) at (-6,-4) {\scriptsize -$3b$};
\node (1b) at (0,-2) {\scriptsize $1b$};
\node (2b) at (2,-2) {\scriptsize $2b$};
\node (+2t) at (10,1) {\scriptsize +$2t$};
\node (+3t) at (12,2) {\scriptsize +$3t$};
\node (+2b) at (10,-3) {\scriptsize +$2b$};
\node (+3b) at (12,-4) {\scriptsize +$3b$};
		
		\draw[->,double] (-2t) -- (-3t);
		\draw[<-,double] (-2b) -- (-3b);
		\draw[->,double] (+2t) -- (+3t);
		\draw[<-,double] (+2b) -- (+3b);
	\end{tikzpicture}
	\\
	\vspace{1.0em}
	\begin{tikzpicture}[scale=0.6, every node/.style={draw, circle, thick, minimum size=5mm,inner sep=0pt}]
		\node (-1t) at (-2,1) {\scriptsize $-1t$};
\node (-2t) at (-4,1) {\scriptsize -$2t$};
\node (-3t) at (-6,2) {\scriptsize -$3t$};
\node (-4t) at (-6,0) {\scriptsize -$4t$};
\node (-1b) at (-2,-3) {\scriptsize -$1b$};
\node (-2b) at (-4,-3) {\scriptsize -$2b$};
\node (-3b) at (-6,-4) {\scriptsize -$3b$};
\node (-4b) at (-6,-2) {\scriptsize -$4b$};
\node (1t) at (0,0) {\scriptsize $1t$};
\node (1b) at (0,-2) {\scriptsize $1b$};
\node (2t) at (2,0) {\scriptsize $2t$};
\node (2b) at (2,-2) {\scriptsize $2b$};
\node (3t) at (4,0) {\scriptsize $3t$};
\node (3b) at (4,-2) {\scriptsize $3b$};
\node (4t) at (6,0) {\scriptsize $4t$};
\node (4b) at (6,-2) {\scriptsize $4b$};
\node (+1t) at (8,1) {\scriptsize +$1t$};
\node (+2t) at (10,1) {\scriptsize +$2t$};
\node (+3t) at (12,2) {\scriptsize +$3t$};
\node (+4t) at (12,0) {\scriptsize +$4t$};
\node (+1b) at (8,-3) {\scriptsize +$1b$};
\node (+2b) at (10,-3) {\scriptsize +$2b$};
\node (+4b) at (12,-2) {\scriptsize +$4b$};
\node (+3b) at (12,-4) {\scriptsize +$3b$};
		
		\draw[->,red] (-1t) -- (1t);
		\draw[->,red] (1t) -- (1b);
		\draw[<->] (2b) -- (1b);
		\draw[->,red] (1t) -- (2t);
		\draw[->,red] (2b) -- (2t);
		\draw[->,double,red] (2t) -- (3t);
		\draw[->] (3b) -- (2b);
		\draw[->,red] (-1t) -- (-2t);
		\draw[->,double] (-2t) -- (-3t);
		\draw[->,red] (-4t) -- (-1t);
		\draw[->,red] (-3t) -- (-4t);
		\draw[->] (-3t) -- (-1t);
		\draw[->] (-4t) -- (-2t);
		\draw[->,red] (-4b) -- (-4t);
		\draw[<-,red] (-1b) -- (-2b);
		\draw[<-,double] (-2b) -- (-3b);
		\draw[<-] (-4b) -- (-1b);
		\draw[<-] (-3b) -- (-4b);
		\draw[<-,red] (-3b) -- (-1b);
		\draw[<-] (-4b) -- (-2b);
		\draw[->,red] (1b) -- (-1b);
		\draw[->,red] (3t) -- (3b);
		\draw[->] (4t) -- (4b);
		\draw[<-,red] (4t) -- (3t);
		\draw[<->,red] (3b) -- (4b);
		\draw[->,red] (+1t) -- (+2t);
		\draw[->,double] (+2t) -- (+3t);
		\draw[->,red] (+4t) -- (+1t);
		\draw[->,red] (+3t) -- (+4t);
		\draw[->] (+3t) -- (+1t);
		\draw[->] (+4t) -- (+2t);
		\draw[->,red] (+4b) -- (+4t);
		\draw[<-,red] (+1b) -- (+2b);
		\draw[<-,double] (+2b) -- (+3b);
		\draw[<-] (+4b) -- (+1b);
		\draw[<-] (+3b) -- (+4b);
		\draw[<-,red] (+3b) -- (+1b);
		\draw[<-] (+4b) -- (+2b);
		\draw[->,red] (+1t) -- (4t);
		\draw[->,red] (4b) -- (+1b);
	\end{tikzpicture}
	\caption{Molecule and spanning tree after steps 13-15. We apply BR on $(1b,2b)$, (3S3-5G) on $(\text{+}1t,\text{+}4t)$ and $(\text{+}1b,\text{+}4b)$.}
\end{figure}


\newpage
\begin{thebibliography}{Cha10}
	\expandafter\ifx\csname url\endcsname\relax
	\def\url#1{\texttt{#1}}\fi
	\expandafter\ifx\csname urlprefix\endcsname\relax\def\urlprefix{URL }\fi
	\expandafter\ifx\csname href\endcsname\relax
	\def\href#1#2{#2}\fi
	\expandafter\ifx\csname burlalt\endcsname\relax
	\def\burlalt#1#2{\href{#2}{\texttt{#1}}}\fi
	
	
	
	
	
	\bibitem{BDNY24}
	{\rm B.~Bringmann Y.~Deng, A.~Nahmod, H.~Yue }.
	\newblock \emph{Invariant Gibbs measures for the three dimensional cubic nonlinear wave equation}.
	\newblock  Invent. Math. \textbf{236}, (2024),  1133–1411.
	\newblock
	\burlalt{doi:10.1007/s00222-024-01254-4}{https://dx.doi.org/10.1007/s00222-024-01254-4}.
	
	
	
	
		\bibitem{BGHS21}
	{\rm T. Buckmaster, P. Germain, Z. Hani, J. Shatah}.
	\newblock \emph{Onset of the wave turbulence description of the longtime behavior of the nonlinear Schroedinger equation}.
	\newblock  Invent. math. \textbf{225}, (2021),  787–855.
	\newblock
	\burlalt{doi:10.1007/s00222-021-01039-z}{https://dx.doi.org/10.1007/s00222-021-01039-z}.
	
\bibitem{BGSS22}	T. Bodineau, I. Gallagher, L. Saint-Raymond, S. Simonella. \emph{Cluster expansion for a dilute hard sphere gas dynamics.}
	J. Math. Phys. \textbf{63}, no. 7, (2022), Paper No. 073301.
		\burlalt{doi:10.1063/5.0091199}{https://dx.doi.org/10.1063/5.0091199}.
		
		
	\bibitem{BGSS23} T. Bodineau, I. Gallagher, L. Saint-Raymond, S. Simonella. \emph{Statistical dynamics of a hard sphere gas: fluctuating
	Boltzmann equation and large deviations.} Ann. of Math. (2) \textbf{198}, no. 3, (2023), 1047–1201.
	\burlalt{doi:10.4007/annals.2023.198.3.3}{https://dx.doi.org/10.4007/annals.2023.198.3.3}.
	
		\bibitem{BP57}
	N.~N. Bogoliubow, O.~S. Parasiuk.
	\newblock { \em \"{U}ber die {M}ultiplikation der {K}ausalfunktionen in der
		{Q}uantentheorie der {F}elder.}
	\newblock Acta Math. \textbf{97}, (1957), 227--266.
	\newblock
	\burlalt{doi:10.1007/BF02392399}{http://dx.doi.org/10.1007/BF02392399}.
	
	\bibitem{CLRS22} T. H. Cormen, , C. E. Leiserson, R. L. Rivest, S. Clifford. 
	 \emph{Introduction to Algorithms, fourth edition 4th Edition}. The MIT Press, 4th Edition, (2022).
	
	
	
	\bibitem{DH21}
	{\rm Y.~Deng,  Z.~Hani}.
	\newblock \emph{On the derivation of the wave kinetic equation for NLS}.
	\newblock Forum Math. Pi, (2021), \textbf{9}, e6, (2021), 1–37.
	\newblock
	\burlalt{doi:10.1017/fmp.2021.6}{https://dx.doi.org/10.1017/fmp.2021.6}.
	
	\bibitem{DH23}
	{\rm Y.~Deng,  Z.~Hani}.
	\newblock \emph{Full derivation of the wave kinetic equation}.
	\newblock  Invent. math. \textbf{233}, (2023),  543-724.
	\newblock
	\burlalt{doi:10.1007/s00222-023-01189-2}{https://dx.doi.org/10.1007/s00222-023-01189-2}.
	
	\bibitem{DH2301}
	Y.~Deng, Z.~Hani.
	\newblock {\textsl{Derivation of the wave kinetic equation: Full range of scaling laws.}} Ta appear in Mem. Amer. Math. Soc.
	\newblock \burlalt{arXiv:2301.07063}{http://arxiv.org/abs/2301.07063}.
	
	\bibitem{DH2311}
	Y.~Deng, Z.~Hani.
	\newblock {\textsl{Long time justification of wave turbulence theory.}}
	\newblock \burlalt{arXiv:2311.10082}{http://arxiv.org/abs/2311.10082}.
	

	
	\bibitem{DH26}
	{\rm Y.~Deng,  Z.~Hani}.
	\newblock \emph{Propagation of chaos and higher order statistics in wave kinetic theory}.
	\newblock  J. Eur. Math. Soc. (JEMS), \textbf{28}, no. 2, (2026), 673-733. 
	\burlalt{doi:10.4171/JEMS/1488}{https://dx.doi.org/10.4171/JEMS/1488}.
	
		\bibitem{DHM25}
	Y.~Deng, Z.~Hani, X. Ma.
	\newblock {\textsl{Long time derivation of the Boltzmann equation from hard sphere dyamics.}} To appear in Ann. of Math.
	\newblock \burlalt{arXiv:2408.07818}{http://arxiv.org/abs/2408.07818}.
	
	\bibitem{DHM25}
	Y.~Deng, Z.~Hani, X. Ma.
	\newblock {\textsl{Hilbert's sixth problem: derivation of fluid equations via Boltzmann's kinetic theory.}}
	\newblock \burlalt{arXiv:2503.01800}{http://arxiv.org/abs/2503.01800}.
	
	
	
	
	
	
	
	
	\bibitem{DST25}
A.-S. De Suzzoni, A. Stingo, A. Touati.
	 {\textsl{	Wave turbulence for a semilinear Klein-Gordon system.}}
	\newblock \burlalt{arXiv:2503.24222}{http://arxiv.org/abs/2503.24222}.
	
		\bibitem{ESY07}
	L. Erdös, M. Salmhofer, H.-T. Yau. \emph{Quantum diffusion of the random Schrodinger evolution in the scaling limit II. The recollision diagrams.} Commun. Math. Phys.
	\textbf{271}, (2007) 1–53.
	\burlalt{doi: 10.1007/s00220-006-0158-2}{https://dx.doi.org/ 10.1007/s00220-006-0158-2}.
	
	
		\bibitem{ESY08}
	L. Erdös, M. Salmhofer, H.-T. Yau. \emph{Quantum diffusion of the random Schrödinger evolution in the scaling limit.} Acta
	Math. \textbf{200}, no. 2, (2008), 211–277.
	\burlalt{doi:10.1007/s11511-008-0027-2}{https://dx.doi.org/10.1007/s11511-008-0027-2}.
	
	\bibitem{GNI85I}
{\rm 	G. Gallavotti, F. Nicolò}.
\emph{Renormalization theory in four-dimensional scalar fields (I)}.
\newblock  Commun. Math. Phys. \textbf{100}, (1985), 545–590.
\newblock
\burlalt{doi:10.1007/BF01217729}{https://dx.doi.org/10.1007/BF01217729}.

\bibitem{GNI85II}
{\rm 	G. Gallavotti, F. Nicolò}.
\emph{Renormalization theory in four-dimensional scalar fields (II)}.
\newblock  Commun. Math. Phys. \textbf{101}, (1985), 247–282.
\newblock
\burlalt{doi:10.1007/BF01218761}{https://dx.doi.org/10.1007/BF01218761}.

\bibitem{H69}
K. ~Hepp.
\newblock {\em On the equivalence of additive and analytic renormalization}.
\newblock Comm. Math. Phys. \textbf{14}, (1969), 67--69.
\newblock \burlalt{doi:10.1007/
	BF01645456}{http://dx.doi.org/10.1007/
	BF01645456}.
	
\bibitem{HQ18}
{\rm 	M. Hairer, J. Quastel}.
\emph{A class of growth models rescaling to KPZ}.
\newblock  Forum math. Pi  \textbf{6}, e3, (2018), 1-112.
\newblock
\burlalt{doi:10.1017/fmp.2018.2}{https://dx.doi.org/10.1017/fmp.2018.2}.

\bibitem{Kru56}
{\rm 	J. B. Kruskal}.
\emph{On the shortest spanning subtree of a graph and the traveling salesman
	problem}.
\newblock  Proc. Amer. Math. Soc. \textbf{7}, (1956), 48-50.
\newblock
\burlalt{doi:10.1090/S0002-9939-1956-0078686-7}{https://dx.doi.org/10.1090/S0002-9939-1956-0078686-7}.


\bibitem{Land75}
{\rm 	O.E. Landford}.
\emph{Time evolution of large classical systems}.
Lect. Notes in Physics, Springer Verlag, \textbf{38}, J. Moser ed., 1–111, 
(1975).
\burlalt{doi:10.1007/3-540-07171-7_1}{https://dx.doi.org/10.1007/3-540-07171-7_1}.
	
	\bibitem{LS11}
	{\rm 	J. Lukkarinen,  H. Spohn}.
	 \emph{Weakly nonlinear Schrödinger equation with random initial data}.
	\newblock  Invent. math. \textbf{183}, (2011),  79–188.
	\newblock
	\burlalt{doi:10.1007/s00222-010-0276-5}{https://dx.doi.org/10.1007/s00222-010-0276-5}.
	
	\bibitem{Pr57}
	{\rm 		R. C. Prim}.
	\emph{Shortest connection networks and some generalizations}.
	\newblock  Bell System Technical Journal, \textbf{36}, no. 6, (1957),  1389–1401.
	\newblock
	\burlalt{doi:10.1002/j.1538-7305.1957.tb01515.x}{https://dx.doi.org/10.1002/j.1538-7305.1957.tb01515.x}.
	
	\bibitem{PY17}
	{\rm 		A. Procacci}.
	\emph{A Correction to a Remark in a Paper by Procacci and Yuhjtman: New Lower Bounds for the Convergence Radius of the Virial Series}.
	\newblock Journal of Statistical Physics, \textbf{168}, no. 6, (2017), 1353--1362.
	\newblock
	\burlalt{doi:10.1007/s10955-017-1853-4}{https://dx.doi.org/10.1007/s10955-017-1853-4}.

	\bibitem{RiV91}
	{\rm V. Rivasseau}.
	\emph{From Perturbative to Constructive Renormalization}.
	\newblock  Princeton University Press, (1991).
	\newblock
	\burlalt{doi:10.1007/s00222-010-0276-10.1515/9781400862085}{https://dx.doi.org/10.1515/9781400862085}.
	
	\bibitem{RW14}
	{\rm V. Rivasseau, Z. Wang}.
	\emph{How to Resum Feynman Graphs}.
	\newblock  Ann. Henri Poincaré \textbf{15}, (2014), 2069-2083.
	\newblock
	\burlalt{doi: 10.1007/s00023-013-0299-8}{https://dx.doi.org/ 10.1007/s00023-013-0299-8}.

	
	\bibitem{Z69}
	W. ~Zimmermann.
	\newblock{\em Convergence of Bogoliubov’s method of renormalization in momentum space.}
	\newblock {Comm. Math. Phys. \textbf{15}, (1969), 208--234.}
	\newblock \burlalt{doi:10.1007/BF01645676}{http://dx.doi.org/10.1007/BF01645676}.

	
\end{thebibliography}
\end{document}